\documentclass[10pt]{amsart}
\usepackage[english]{babel}
\usepackage{amssymb}
\usepackage{enumitem}
\usepackage{hyperref}
\usepackage[initials]{amsrefs}
\usepackage[all]{xy}
\SelectTips{cm}{}

\newcommand{\bbN}{{\mathbb N}}
\newcommand{\bbQ}{{\mathbb Q}}
\newcommand{\bbR}{{\mathbb R}}
\newcommand{\bbZ}{{\mathbb Z}}
\newcommand{\bbC}{{\mathbb C}}
\newcommand{\bbH}{\mathbb{H}}

\newcommand{\calW}{\mathcal{W}}
\newcommand{\calG}{\mathcal{G}}

\newcommand{\bs}{\backslash}

\newcommand{\id}{\operatorname{id}}
\newcommand{\im}{\operatorname{im}}

\newcommand{\isom}{\operatorname{Isom}}

\newcommand{\pr}{\operatorname{pr}}

\newcommand{\supp}{\operatorname{supp}}

\newcommand{\conv}{\operatorname{conv}}

\newcommand{\diam}{\operatorname{diam}}

\newcommand{\Hom}{\operatorname{hom}^b}

\newcommand{\rmL}{{L}}

\newcommand{\abs}[1]{{\left\lvert #1\right\rvert}}

\newcommand{\norm}[1]{{\left\lVert #1\right\rVert}}

\newcommand{\rmH}{H}
\newcommand{\rmHb}{H_b}
\newcommand{\rmHl}{H^{(1)}}

\newcommand{\rmC}{C}
\newcommand{\rmCb}{C_{b}}

\newtheorem{theorem}{Theorem}[section]
\newtheorem{lemma}[theorem]{Lemma}

\newtheorem{corollary}[theorem]{Corollary}

\theoremstyle{definition}
\newtheorem{definition}[theorem]{Definition}

\newtheorem{example}[theorem]{Example}
\newtheorem{question}[theorem]{Question}
\newtheorem{remark}[theorem]{Remark}

\numberwithin{equation}{section}
\begin{document}

\title[Efficient subdivision in hyperbolic groups]{Efficient subdivision in hyperbolic groups and applications}

\author{Uri Bader}
\address{Technion, Haifa}
\email{uri.bader@gmail.com}

\author{Alex Furman}
\address{University of Illinois at Chicago, Chicago}
\email{furman@math.uic.edu}

\author{Roman Sauer}
\address{WWU M\"unster, M\"unster}
\email{sauerr@uni-muenster.de}

\thanks{U.B. and A.F. were supported in part by the BSF grant 2008267.}
\thanks{U.B was supported in part by the ISF grant 704/08.}
\thanks{A.F. was supported in part by the NSF grants DMS 0604611 and 0905977.}
\thanks{R.S. gratefully acknowledges support from the \emph{Deutsche Forschungsgemeinschaft},
made through grant SA 1661/1-2, during the initial phase of this project.}


\maketitle

\begin{abstract}
	We identify the images of the comparision maps from ordinary homology and
	Sobolev homology, respectively, to the $\ell^1$-homology of a word-hyperbolic
	group with coefficients in complete normed modules. The underlying idea
	is that there is a subdivision procedure
	for singular chains in negatively curved spaces that is much more
	efficient (in terms of the $\ell^1$-norm) than
	barycentric subdivision. 	
	The results of this paper
	are an important ingredient in
	a forthcoming proof of the authors that
	hyperbolic lattices in dimension $\ge 3$ are rigid with respect to integrable
	measure
	equivalence. Moreover, we prove
	a proportionality principle for the simplicial volume of negatively curved
	manifolds with regard to integrable measure equivalence.
\end{abstract}

\section{Introduction and Statement of the Main Results}
\label{sec:introduction}

Bounded cohomology of (discrete or continuous) groups proved to
be a useful tool for various questions about rigidity of groups. Since bounded
cohomology is, in general, extremely hard to compute, the question of surjectivity or
bijectivity of the comparision map from bounded cohomology to ordinary cohomology
is very important. It is conjectured to be surjective
(and might even by an isomorphism, for all we know) for simple connected
Lie groups with finite center~\cite{conj} and trivial coefficients.
Furthermore, it is surjective for (discrete) word-hyperbolic groups and arbitrary
Banach modules as coefficients~\cite{mineyev}.

In this paper we are concerned with a kind of pre-dual situation:
the comparision map from the
ordinary homology of a discrete group into its $\ell^1$-homology. We prove
in Theorem~\ref{hyp-sob} that for
word-hyperbolic groups the image of the comparision map in the $\ell^1$-homology
coincides with the image of a similar comparision map from
\emph{Sobolev homology} (Definition~\ref{def:various homologies})
to the $\ell^1$-homology.

The Sobolev chain complex $C_\ast^{(1,1)}(G)$
of a group $G$
can be viewed as a subcomplex
of the $\ell^1$-chain complex $C_\ast^{(1)}(G)$, containing the
ordinary chain complex $C_\ast(G)$, but being much larger than $C_\ast(G)$.
The Sobolev homology of a group $G$ with coefficients in the Banach space
$\rmL^1(X)$, where $X$ is a $G$-probability space, is a natural recipient of
certain maps associated to measure equivalence cocycles $G\times X\to G$ satisfying an
$\ell^1$-condition. This is reminiscient of the work of Monod and
Shalom~\cite{monod+shalom}, where bounded cohomology is used as a recipient
of certain maps associated to measure equivalence
cocycles, and where they prove
rigidity results regarding measure equivalence of products of word-hyperbolic groups.
In a forthcoming paper~\cite{mostow+me} we prove
that hyperbolic lattices are rigid with respect to integrable measure equivalence
using the main result of this paper, Theorem~\ref{hyp-sob}; in the present paper
we give an application to the simplicial volume: in
Theorem~\ref{thm:invariance of simplicial volume} we
prove a proportionality principle with regard to integrable measure equivalence
for fundamental groups of closed negatively curved manifolds, which generalizes
Gromov's proportionality principle~\cite{gromov}*{section~0.4} for such manifolds.

\subsection{Norms on the standard resolution and Sobolev homology} % (fold)
\label{sub:norms_on_the_standard_resolution_and_sobolev_homology}

Let $X$ be a set. We consider the
chain complex $C_\ast(X)$ where $C_n(X)$ is the free abelian group with
basis $X^{n+1}$ and differentials defined by
\[ d_n(x_0,x_1,\ldots,x_n)=\sum_{i=0}^n (-1)^i(x_0,\ldots,\widehat{x_i},\ldots,x_n). \]
If $X=G$ is a group, then $C_\ast(G)$ is called
the \emph{standard homogeneous resolution} of $G$.
Endowing each $C_n(G)$ with the diagonal $G$-operation, $C_\ast(G)$ becomes a
chain complex of $\bbZ G$-modules. Let $C_\ast(G,\bbR)=C_\ast(G)\otimes_\bbZ\bbR$
be the
corresponding complex with real coefficients.

There is a variety of norms one might impose on $C_\ast(G)$; we consider the
following:

\begin{definition}\label{def:norms}
Let $G$ be finitely generated. Fix a word metric on $G$. For a subset
$S\subset G$ we denote by $\diam(S)$ the diameter with respect to this word metric.
On $C_n(G)$ and $C_n(G,\bbR)$ we define
\begin{enumerate}
\item the \emph{$\ell^1$-norm} $$\|\sum a_{(g_0,g_1,\ldots,g_n)}\cdot (g_0,g_1,\ldots,g_n)\|_1= \sum |a_{(g_0,g_1,\ldots,g_n)}|,$$
\item and the \emph{Sobolev norm}
\begin{align*}
\|\sum a_{(g_0,g_1,\ldots,g_n)}&\cdot (g_0,g_1,\ldots,g_n)\|_{1,1}\\&= \sum |a_{(g_0,g_1,\ldots,g_n)}|\cdot\bigl(1+\diam(g_0,g_1,\ldots,g_n)\bigr).
\end{align*}
\end{enumerate}
\end{definition}

It is easy to verify that different word metrics on $G$ give rise to equivalent
Sobolev norms. We denote $C_n(G)$ when endowed with the $\ell^1$-norm or Sobolev norm
by $C_n^{(1,1)}(G)$ and $C_n^{(1)}(G)$ respectively. Note that
both are automatically complete because of the integral coefficients.
The differentials above are continuous with
respect to both the $\ell^1$-norm and the Sobolev norm. In particular,
we obtain chain complexes of normed modules
$C_\ast^{(1,1)}(G)$ and $C_\ast^{(1)}(G)$.

In section~\ref{sec:normed rings} we explain the less common setting of
normed rings and normed modules over normed rings. The integral group
ring $\bbZ G$ endowed with the $\ell^1$-norm is a normed ring in the sense
of definition~\ref{def:normed ring}. The chain complexes $C_\ast^{(1,1)}(G)$ and
$C_\ast^{(1)}(G)$ are normed chain complexes over the normed ring $\bbZ G$. In
subsection~\ref{sub:tensor products} we explain the construction of a completed
tensor product $\widehat\otimes_{\bbZ G}$ over the normed ring $\bbZ G$ endowed with
the $\ell^1$-norm, which is an integral version of the projective
tensor product of Banach spaces (compare Remark~\ref{rem:recover projective tensor product}). It will be essential to the proof of
Theorem~\ref{thm:invariance of simplicial volume} to use  $\widehat\otimes_{\bbZ G}$
rather than the usual projective tensor product.

\begin{definition}\label{def:various homologies}
	Let $E$ be a normed module $E$ over the normed ring $\bbZ G$.
	The $n$-th \emph{Sobolev homology} $H^{(1,1)}_n(G,E)$ is the $n$-th homology of
	$C_\ast^{(1,1)}(G)\widehat{\otimes}_{\bbZ G} E$. The $n$-th $\ell^1$-homology
	$H_n^{(1)}(G,E)$
	is the $n$-th homology of $C_\ast^{(1)}(G)\widehat{\otimes}_{\bbZ G} E$.
\end{definition}

\begin{definition}\label{def:comparision map}
	Let $E$ be a normed module $E$ over the normed ring $\bbZ G$.
	The homomorphisms
	$H_\ast(G,E)\to H_\ast^{(1)}(G,E)$ and
	$H_\ast^{(1,1)}(G,E)\to H_\ast^{(1)}(G,E)$ induced by the
	natural chain maps
	$C_\ast(G)\otimes_{\bbZ G} E\to C_\ast^{(1)}(G)\widehat\otimes_{\bbZ G} E$
	(compare Lemma~\ref{lem:universal property tensor}) and
	$C_\ast^{(1,1)}(G)\widehat\otimes_{\bbZ G}E
	\to C_\ast^{(1)}(G)\widehat\otimes_{\bbZ G} E$ (compare
	Example~\ref{exa:inclusion maps}), respectively, are
	called \emph{comparision maps}.
\end{definition}

\begin{remark}
	The Sobolev homology (or rather its dual) is reminiscient of the
	so-called \emph{group cohomology with polynomial growth}, which was
	studied by Connes and Moscovici in the context of Novikov
	conjecture~\cite{connes}.
\end{remark}

\begin{remark}\label{rem:semi-norms on homology}
	If $E_\ast$ is a chain complex of normed modules over a normed ring, then
	its homology groups $H_n(E_\ast)$ inherit a semi-norm by defining the semi-norm
	of a homology class $x$ as the infimum of the norms of chains representing $x$.
	
	If $E$ is
	a Banach space with isometric $G$-action, then $H_n^{(1)}(G,E)$ is just the usual
	$\ell^1$-homology endowed with the usual semi-norm (see also
	Remark~\ref{rem:recover projective tensor product}).
\end{remark}

\subsection{Main results} % (fold)
\label{sub:main_result}

% subsection main_result (end)
Our main theorem is:

\begin{theorem} \label{hyp-sob}
Let $G$ be a hyperbolic group. Let $E$ be a complete normed $\bbZ G$-module.
Then the following images under the comparision maps coincide:
\[
	\im\Bigl(H_\ast(G, E)\to H_\ast^{(1)}(G,E)\Bigr)=
	\im\Bigl(H_\ast^{(1,1)}(G,E)\to H_\ast^{(1)}(G,E)\Bigr).
\]
\end{theorem}

The above theorem follows rather easily (see section~\ref{sec:proof of main thm})
from the following theorem. Here $C_\ast^r(G)$ denotes the \emph{Rips complex}
of $G$, which is a subcomplex of $C_\ast(G)$
(see Definition~\ref{def: rips complex}).

\begin{theorem}\label{bounded-chian-map}
	Let $G$ be a $\delta$-hyperbolic group.
	There exist a $\bbZ G$-chain homomorphism
	$f_\ast\colon C_\ast(G)\to C_\ast(G)$ and
	constants $r(n)=r(n,\delta)>0$ for every $n\ge 0$ such that:
	\begin{enumerate}
		\item $f_0=\id$,
		\item $\im(f_i)\subset C_i^{r(i)}(G)$ for $i\ge 0$, and
	    \item $f_i$ is bounded with respect to the Sobolev norm on the domain and
	the \mbox{$\ell^1$-norm} on the target for $i\ge 0$.
	\end{enumerate}
\end{theorem}

\begin{remark}
The statement in the preceding theorem is actually true for some constant
$r=r(i)$ that does not depend on $i$ (only on the group $G$),
thus giving a chain map $f_*:C_*(G) \to C_*^r(G)$ which can be seen to be a homotopy equivalence.
We refrain from a proof of this statement since it
is more complicated, and the weaker statement in Theorem~\ref{bounded-chian-map}
is sufficient to conclude
our main result, Theorem~\ref{hyp-sob}, and its corollaries presented below.
\end{remark}

\begin{remark}
The map $f_i$ is a kind of subdivision map that maps arbitrarily large simplices
in $C_i(G)$ to a sum of simplices of bounded diameter (bounded by $r(i)$). For
$f_i$ to be continuous the number of simplices in this sum must grow at
most linearly in the diameter of the large simplex you start with.
That such an efficient subdivision is possible on trees is quite easy to see. We
approximate simplices in $C_i(G)$ by trees (see
Theorem~\ref{thm:rough isometry to a tree}) to
reduce to this case.
\end{remark}

Next we present an application of Theorem~\ref{hyp-sob} to the simplicial volume
of aspherical manifolds, which will be proved
in Section~\ref{sec:_ell_1_measure_equivalence_and_simplicial_volume}.
Recall that a topological space
is \emph{aspherical} if its universal
cover is contractible. Two aspherical CW-complexes are homotopy equivalent
if and only if their fundamental groups are isomorphic.
The
\emph{simplicial volume} $\norm{M}\in\bbR^{\ge 0}$ of an
$n$-dimensional closed orientable manifold $M$
is defined as the infimum of the $\ell^1$-norms
of real singular chains chains representing the fundamental class in $H_n(M,\bbR)$.
The simplicial volume
has many applications in geometry; see
the groundbreaking paper of Gromov~\cite{gromov} for much more information.

The definitions
of \emph{(integrable) measure equivalence} and \emph{(integrable) ME-coupling}
will be recalled in Subsection~\ref{sub:_ell_1_measure_equivalence}.
Measure equivalence is an equivalence relation between groups,
introduced by Gromov in \cite{gromov-invariants} as a measure-theoretic counter part
to quasi-isometry between finitely generated groups; it is intimately related
to orbit equivalence in ergodic theory, to the theory of von Neumann algebras, and to questions in descriptive set theory. We will not go further into a discussion of
this notion here, but refer the reader to the surveys~\cite{shalom-survey, Popa:ICM, Furman:MGT}.

\begin{theorem}\label{thm:invariance of simplicial volume}
	Let $M$ and $N$ be closed, aspherical, orientable manifolds.
	Assume that there
	is an ergodic, integrable ME-coupling $(\Omega,\mu)$ of the fundamental groups
	${G}=\pi_1(M)$ and ${H}=\pi_1(N)$ with coupling
	index $c_\Omega=\mu({H}\bs\Omega)/\mu({G}\bs\Omega)$. Then:
	\begin{enumerate}
		\item If $\norm{N}>0$ and ${G}$ is word-hyperbolic, then $\dim(N)\le \dim(M)$.
		\item Assume that ${G}$ and ${H}$ are
		word-hyperbolic and that $\norm{M}>0$ and $\norm{N}>0$. Then:
		\[\norm{M}=c_\Omega\cdot \norm{N}\text{ and }\dim(M)=\dim(N).\]
	\end{enumerate}
\end{theorem}

Since closed, orientable, negatively curved manifolds have
positive simplicial volume~\cite{gromov}*{0.3~Thurston's theorem} and
word-hyperbolic fundamental groups,
we obtain the following corollary.

\begin{corollary}\label{cor:invariance of simplicial volume for negative curvature}
	Let $M$ and $N$ be closed, orientable, negatively curved manifolds.
	Assume that there
	is an ergodic, integrable ME-coupling $(\Omega,\mu)$ of the fundamental groups
	${G}=\pi_1(M)$ and ${H}=\pi_1(N)$ with coupling
	index $c_\Omega=\mu({H}\bs\Omega)/\mu({G}\bs\Omega)$. Then
	$\norm{M}=c_\Omega\cdot \norm{N}$ and $\dim(M)=\dim(N)$. Further, if ${H}\cong{G}$,
	then $c_\Omega=1$.
\end{corollary}

\begin{remark}\label{rem:ergodicity}
	Any ME-coupling $(\Omega,\mu)$ has a decomposition~\cite{furman}*{Lemma~2.2} into
	ergodic ME-couplings $(\Omega,\mu_z)$. If $(\Omega,\mu)$ is integrable,
	then almost every
	$(\Omega,\mu_z)$ is integrable (see also~\cite{mostow+me}). Hence the
	the equality of dimensions in the previous corollary also holds without the
	ergodicity assumption.
\end{remark}

\begin{remark}\label{rem:recovering proportionality principle}
	Let $M$ and $N$ be closed, orientable, negatively curved manifolds with
	isometric universal covers. Denote their universal cover by $W$. Then the isometry
	group of $W$ contains both $\pi_1(M)$ and $\pi_1(N)$ as cocompact lattices.
	In particular, the isometry group of $W$ endowed with the Haar measure is an
	integrable measure coupling. Up to the ergodicity assumption (which actually
	can be ignored here
	due to the Howe-Moore theorem and various vanishing results for simplicial volume),
	the previous corollary generalizes
	Gromov's proportionality principle~\cite{gromov}*{section~0.4} in that
	situation.
\end{remark}

A positive answer to the following question would be a even more
far-reaching generalization
of the proportionality principle of the simplicial volume with strong consequences for
the measure equivalence rigidity of hyperbolic lattices (see~\cite{mostow+me});
a positive answer would
also fit well with the proportionality of $L^2$-Betti numbers with regard to
measure equivalence~\cite{gaboriau} and a conjectural bound of $L^2$-Betti numbers by the
simplicial volume~\cite{lueck}*{Conjecture~14.1 on p.~489}.

\begin{question}
	Let $M$ and $N$ be closed, orientable, aspherical manifolds. Assume that $\pi_1(M)$ and
	$\pi_1(N)$ are measure equivalent with index $c>0$. Does this imply that
	$\norm{M}=c\cdot \norm{N}$? Further, if both simplicial volumes are positive, are the
	dimensions of $M$ and $N$ equal?
\end{question}

\subsection{Some global conventions} % (fold)
\label{sub:some_conventions}

We use the terms \emph{hyperbolic group} and \emph{word-hyperbolic group}
interchangeably. We also use the terms \emph{integrable ME-coupling} and
\emph{$\ell^1$-ME-coupling} interchangeably.
A \emph{$\delta$-hyperbolic group} is understood in
the sense of~\cite{bridson+haefliger}*{Definition~1.1 on p.~399} using the
slim triangles condition.

We denote the metric on a metric space generically by $d$; we also denote the differential
in a chain complex generically by $d$, if it is clear without ambiguity.

\section{Normed rings and modules}\label{sec:normed rings}

We transfer several concepts from topological vector spaces to the setting of
\mbox{$R$-mo}dules, where $R$ is an arbitrary normed ring (for example, $R=\bbZ$).
Most of this section is straightforward but we review basic notions like, e.g.,
completions and tensor products
because it is not very common to consider normed modules over
$\bbZ$ or $\bbZ G$.

Let $R$ be a unital ring, and let $1_R$ denote its unit element.
We follow the usual convention and denote the element $n\cdot 1_R$ by $n\in R$.

\begin{definition}\label{def:normed ring}
	Let $|\cdot|_{\bbZ}$ denote the usual absolute value on $\bbZ$.
	A \emph{normed ring} $R$ is a unital ring $R$
	endowed with a real-valued function
	$x\mapsto |x|_R$ (called \emph{norm}) such
	that for all $x,y\in R$ and every $n\in \bbZ$:
	\begin{enumerate}
	\item $|x|_R=0\Leftrightarrow x=0$;
	\item $|x+y|_R\leq |x|_R+|y|_R$;
	\item $|xy|_R \leq |x|_R|y|_R$;
	\item $|nx|_R=|n|_{\bbZ}|x|_R$.
	\end{enumerate}
\end{definition}

\begin{definition}\label{def:normed module}
	A \emph{normed (left) $R$-module} over a normed ring
	$R$ is a (left) \mbox{$R$-mo}\-dule $M$ endowed with a real-valued \emph{norm function}
	$m\mapsto \|m\|_M$ such that for all $u,v\in M$, $r\in R$, and $n\in \bbZ$:
	\begin{enumerate}
		\item $\|u\|_M=0\Leftrightarrow u=0$;
		\item $\|u+v\|_M\leq \|u\|_M+\|v\|_M$;
		\item $\|ru\|_M \leq |r|_R\|u\|_M$;
		\item $\|nu\|_M=|n|_{\bbZ}\|u\|_M$;
	\end{enumerate}
	A normed right $R$-module is defined similarly.
\end{definition}

If $\norm{\cdot}$ on $M$ satisfies (2)-(4),
but not necessarily (1), we call
$M$ a \emph{semi-normed module}.
Whenever $M$ or $R$ are clear from the context,
we denote the norms on $M$ or $R$ simply by $\norm{\cdot}$ or $\abs{\cdot}$, respectively.
Observe that a normed module is necessarily torsion-free as an abelian group.

A \emph{normed complex} over $R$ is just a chain complex in the category of normed
$R$-modules. A \emph{bounded $R$-chain map} is a chain map between
normed complexes consisting of bounded $R$-homomorphisms in each degree.

\begin{example}\label{exa:inclusion maps}
	Let $\bbZ G$ be endowed with the $\ell^1$-norm. Then $\bbZ G$ is a normed ring.
    The $\bbZ G$-modules $C_k(G)$ are normed modules when endowed with either the $\ell^1$ or the Sobolev norms
    (Definition~\ref{def:norms}).
	To emphasize the normed module structure on these modules we will denote them by $C_k^{(1)}(G)$ and $C_k^{(1,1)}(G)$, respectively.
    The boundary maps $d_k$ are bounded with respect to both norms (by $k+1$), thus both complexes
    $C_*^{(1)}(G)$ and $C_*^{(1,1)}(G)$ are normed $\bbZ G$-complexes.
    The inclusion $C_*^{(1,1)}(G) \hookrightarrow C_*^{(1)}(G)$ is a bounded chain map of norm at most~$1$.
\end{example}

\subsection{Dual spaces and completions}

Let $R$ be a normed ring.
A homomorphism $\phi:M\to N$ between
two normed $R$-modules $M,N$ is continuous with respect to the topologies induced by the
norm of $M$ and $N$, respectively, if and only if it is bounded, that is, there is
$c\geq 0$ such that $\norm{\phi(m)}\le c\norm{m}$; the infimum of such constants $c$ is the
\emph{operator norm} $\norm{\phi}$.
In that case we say that $\phi$ is a \emph{bounded $R$-homomorphism}.

Let $\Hom_R(M,N)$ be the abelian group of bounded $R$-homomorphisms from $M$ to $N$.
Equipped with the operator norm it becomes a
normed $\bbZ$-module.
Every normed $R$-module has an underlying normed $\bbZ$-module. So we can define
its dual
\[M'=\Hom_\bbZ(M,\bbR).\]
If $M$ is a normed left $R$-module, then $M'$ is naturally a normed right $R$-module, and the double dual $M''$ is again a normed left $R$-module.
In fact $M'$ and $M''$ are real vector spaces (and modules over $\bbR \otimes_{\bbZ} R$).

Given a Cauchy sequence $(f_i)$ in $M'$, one verifies that
$f(m)=\lim_{i\to \infty} f_i(m)$ defines a bounded $f\in M'$. Hence we obtain:

\begin{lemma}\label{lem:dual is complete}
	$M'$ is complete.
\end{lemma}

There is a natural map $M\to M''$, given, as usual, by an evaluation.
The following is a version of the Hahn-Banach theorem that applies for
normed modules.

\begin{lemma}[Hahn-Banach for normed modules]\label{lem:hahn-banach-lemma}
Let $R$ be a normed ring, and $M,N$ be semi-normed $R$-modules.
\begin{enumerate}
\item
For an injective bounded $R$-homomorphism $N\hookrightarrow M$,
the induced dual map $N'\to M'$ is surjective.
\item
Every $m\in M$ has a supporting functional, that is,
\[ \forall m\in M~\exists f\in M' \text{ such that } \|f\|=1\text{ and }\|f(m)\|=\|m\|.\]
\item
The canonical bounded $R$-homomorphism $M\to M''$
given by evaluation is isometric. If $M$ is normed, it is also injective.
\end{enumerate}
\end{lemma}

\begin{proof}
Firstly, since the $R$-linearity in the above statement is automatic, we
regard $M,N$ as $\bbZ$-modules. Secondly,
observe that it is enough to prove (1).
Indeed, assertion~(2) implies assertion~(3), and
assertion~(1) implies (2) by setting $N=\bbZ m$, and letting $f$ be an
extension of the map $N\to \bbR$ induced by $m\mapsto \|m\|$.

Assertion~(1) can be easily reduced to the Hahn-Banach theorem for
$\bbQ$-vector spaces and $\bbR$-valued functionals. Although the
Hahn-Banach theorem for $\bbQ$-vector spaces is not commonly stated,
the usual proof for real vector spaces (see e.g.~\cite{conway}*{III~\S 6})
carries over verbatim.

The map $M\to \bbQ\otimes_\bbZ M,~m\mapsto 1\otimes m$ is an injection since
$M$ is torsion-free.
The following norm on $\bbQ\otimes_\bbZ M$ is the unique one that turns this
injection into an isometry: Let $a_i, b_i$ integers for $i=1,\ldots, m$. Let
$c=b_1b_2\cdots b_n$. Then we set
\[
	\bigl\|\sum_{i=1}^m a_i/b_i\otimes m_i\bigr\|_{\bbQ\otimes_\bbZ M}=c^{-1}\bigl\|\sum_{i=1}^m \frac{a_ic}{b_i}m_i\bigr\|_M.
\]
The isometric embedding $M\hookrightarrow\bbQ\otimes_\bbZ M$ induces an isometric
isomorphism $(\bbQ\otimes M)'\rightarrow M'$. Its inverse is given by
\[ M' \ni f\mapsto \bar{f},~~ \bar{f}(q\otimes m)=q f(m). \]

The proof of (1) now follows:
for an injection of $\bbZ$-normed modules, $N\hookrightarrow N$,
we obtain an injection
$\bbQ\otimes N\hookrightarrow \bbQ\otimes M$ which induces
by Hahn-Banach a surjection $(\bbQ\otimes M)'\twoheadrightarrow (\bbQ\otimes N)'$, thus a surjection $M' \twoheadrightarrow N'$.
\end{proof}

The \emph{completion} $\bar{M}$ of a (semi-)normed $R$-module $M$
is defined as the closure of the image of $M$ in $M''$.
Using Lemma~\ref{lem:hahn-banach-lemma} (3) one easily sees that
the completion satisfies the following universal
property: Every bounded homomorphism of $M$ into a complete normed
$R$-module $V$ extends uniquely to a bounded homomorphism from $\bar{M}$ to $V$.

\subsection{Tensor products}\label{sub:tensor products}

Our next goal is to define the tensor product of normed modules.
Our definition below is an extension of the construction known as the
\emph{projective tensor product}, which satisfies a universal property with respect to bilinear maps.

Given a normed right $R$-module $E$, a normed left $R$-module $F$, a normed $\bbZ$-module $V$
and a $\bbZ$-module morphism $\phi:E\otimes_R F\to V$, we obtain the associated $R$-bilinear map
$\tilde{\phi}:E\times F \to V$.
We set
\[ \|\tilde{\phi}\|=\inf\bigl\{c\geq 0~|~ \forall~ f\in F,~e\in E,~\|\tilde{\phi}(f,e)\|\leq c\|f\|\|e\|\bigr\}, \]
and say that $\tilde{\phi}$ is \emph{bounded} if $\|\tilde{\phi}\|<\infty$.

\begin{lemma}\label{lem:universal property tensor}
Let $E$ be a normed right $R$-module and $F$ be a normed left $R$-module.
Then there is a complete normed $\bbZ$-module, denoted by $E\widehat{\otimes}_R F$,
and a $\bbZ$-homomorphism $E\otimes_R F\overset{p}{\to} E\widehat{\otimes}_R F$
satisfying the following universal property:
\[
\xymatrix{E\otimes_R F \ar[r]^{\phi}\ar[d]^p & V \\
          E\widehat{\otimes}_R F \ar@{-->}[ur]_{\exists !\bar{\phi}} &}
\]
In words: For every complete $\bbZ$-module $V$
and for every $\bbZ$-homomorphism
$\phi:E\otimes_R F\to V$ such that the
associated bilinear map $\tilde{\phi}:E\times F \to V$ is bounded,
there exist a unique bounded
$\bbZ$-homomorphism $\bar{\phi}:E\widehat{\otimes}_R F \to V$ such that $\phi=\bar{\phi}\circ p$
and $\|\bar{\phi}\|=\|\tilde{\phi}\|$.

Furthermore, the pair $(E\widehat{\otimes}_R F,p)$ is unique up to
isometric isomorphism.
\end{lemma}

\begin{definition}\label{def:projective tensor product}
	Retain the setting of the previous lemma.
	The normed module $E\widehat{\otimes}_R F$ is called the {\em projective tensor
	product} of $E$ and $F$ over $R$.
\end{definition}

\begin{proof}
There is a natural $\bbZ$-module morphism:
\[ E\otimes_R F {\to} \Hom_R(E,F')',\quad e\otimes f\mapsto\bigl(T \mapsto \langle Te,f\rangle\bigr), \]
where $\langle\_,\_\rangle$ is the evaluation map $F'\times F\to\bbR$.
We denote the closure of the image by $E\widehat{\otimes}_R F$ and the map
$E\otimes_R {\to} E\widehat{\otimes}_R F$ by $p$.
By Lemma~\ref{lem:dual is complete},
$\Hom_R(E,F')'$ is complete, hence so is $E\widehat{\otimes}_R F$.

For every complete $\bbZ$-module $V$ and for every $\bbZ$-homomorphism
$\phi:E\otimes_R F\to V$ such that the associated bilinear
map $\tilde{\phi}:E\times F \to V$ is bounded, we obtain
the map
\[ V' \xrightarrow{\psi} \Hom_R(E,F'), \quad \psi(v')= e\mapsto \bigl(f \mapsto v'(\tilde{\phi}(e,f))\bigr), \]
and it is clear that the composition
\[ E\otimes_R F \xrightarrow{p} \Hom_R(E,F')' \xrightarrow{\psi'} V''\]
coincides with
\[ E\otimes_R F \xrightarrow{\phi} V \xrightarrow{i} V'',\]
where $i$ is the canonical map given by evaluation.
Since $V$ is complete, $i(V)$ is closed in $V''$ by
Lemma~\ref{lem:hahn-banach-lemma}, and therefore $\psi'^{-1}(i(V))$ is
closed in $\Hom_R(E,F')'$.
It follows that the closure of $p(E\otimes_R F)$, that is $E\widehat{\otimes}_R F$,
lies in $\psi'^{-1}(i(V))$.
Therefore $\psi'(E\widehat{\otimes}_R F)\subset i(V)\cong V$, and we obtain a map $\bar{\phi}:E\widehat{\otimes}_R F\to V$
such that $\bar{\phi}\circ p=\phi$.
We leave it to the reader to check using Lemma~\ref{lem:hahn-banach-lemma}
that indeed $\|\bar{\phi}\|=\|\tilde\phi\|$.

The uniqueness of the pair $(E\widehat{\otimes}_R F,p)$ up to isomorphism follows
directly from the universal property.
Observe that by choosing the above $\phi$ to be the identity map
of $E\widehat{\otimes}_R F$ we get that $\|\tilde{p}\|=1$.
It follows that the unique isomorphism between modules
having the above universal property is actually isometric.
\end{proof}

We summarize some of the properties of the projective tensor product.
The proofs are easy and use the universal property above; we leave them
to the reader.

\begin{lemma}  \label{cor:tensor}
Let $M$ be an normed $R$-module. The following isomorphisms are natural and
isometric:
\begin{enumerate}
\item
$R \widehat{\otimes}_R M$ is isomorphic to the completion $\bar{M}$.
In particular, $\bbZ \widehat{\otimes}_{\bbZ} M$ is isomorphic to $\bar{M}$.
\item
$\bbR \widehat{\otimes}_{\bbZ} M$ is a Banach space.
If $M$ is a normed real vector space, $\bbR \widehat{\otimes}_{\bbZ} M$ is
isomorphic to $\bar{M}$.
If $M$ is a Banach space, $\bbR \widehat{\otimes}_{\bbZ} M$ is isomorphic to $M$.
\item
We have
$M'\cong (\bbR \widehat{\otimes}_{\bbZ} M)'
\cong \Hom_\bbR(\bbR \widehat{\otimes}_{\bbZ} M,\bbR)$.
In particular, if $M$ is a Banach space, then $M'$ is
isomorphic to the dual of $M$ as a Banach space.
\end{enumerate}
\end{lemma}

The functor $M\mapsto \bbR \widehat{\otimes}_{\bbZ} M$ is called {\em Banachification}.

\begin{example}\label{exa:L1(X,Z) example}
Let	$(X,\mu)$ be a measure space.
Let $M$ be the abelian group consisting of finite-valued measurable functions from $X$ to $\bbZ$ supported on a set of finite measure.
Integration gives a semi-norm on $M$, turning it into a semi-normed $\bbZ$-module.
The completion $\bar{M}$ is denoted by $\rmL^1(X,\bbZ)$.
It is the normed module consisting of
$\mu$-integrable measurable maps $f\colon X\to\bbZ$ modulo null sets.
The Banachification of $M$ is naturally identified with $\rmL^1(X,\bbR)$.
The dual space $M'$ is isomorphic to the Banach dual of $\rmL^1(X,\bbR)$, hence can be identified with
$\rmL^\infty(X,\bbR)$.

Let $(Y,\nu)$ be another measure space. Let $\pr_X\colon X\times Y\to X$ and
$\pr_Y\colon X\times Y\to Y$ be the projections onto $X$ and $Y$,
respectively. The isometric bilinear map
\[
	L^1(X,\bbZ)\times L^1(Y,\bbZ)\to L^1(X\times Y,\bbZ),~(f,g)\mapsto (f\circ\pr_X)\cdot (g\circ\pr_Y),
\]
where $X\times Y$ carries the product measure $\mu\times\nu$,
induces, by the universal property, an isometric map
\[ L^1(X,\bbZ) \widehat{\otimes}_{\bbZ}  L^1(Y,\bbZ) \xrightarrow{\cong}  L^1(X \times Y,\bbZ). \]
Since the image is closed and dense, this map is an isometric isomorphism.
Similarly, we obtain an isometric isomorphism
\[ L^1(X,\bbR) \widehat{\otimes}_{\bbR}  L^1(Y,\bbR) \xrightarrow{\cong}  L^1(X \times Y,\bbR). \]
\end{example}

\begin{example} \label{ex:L1-C}
Taking in the previous example $Y=G^{k+1}$ endowed with
the counting measure, we obtain the isomorphisms
\begin{align*}
	C^{(1)}_k(G) \widehat{\otimes}_{\bbZ} L^1(X,\bbZ) &\xrightarrow{\cong} L^1(G^{k+1}\times X,\bbZ)\\
 	C^{(1)}_k(G) \widehat{\otimes}_{\bbZ} L^1(X,\bbR) &\xrightarrow{\cong} L^1(G^{k+1}\times X,\bbR),
\end{align*}
where $G^{k+1}\times X$ carries the product of the counting measure and the given measure on $X$.

If we endow $Y=G^{k+1}$ with the measure that assigns to each point
$(g_0,\ldots,g_k)$ the weight $1+\diam(g_0,\ldots,g_k)$, then
$L^1(Y)\cong C^{(1,1)}_k(G)$ as normed $\bbZ G$-modules, and we
obtain an isomorphism
\begin{multline*}
C^{(1,1)}_k(G) \widehat{\otimes}_{\bbZ} L^1(X,\bbZ) \xrightarrow{\cong}
 \Bigl\{f\in L^1(G^{k+1}\times X,\bbZ)~|\\\int_{G^{k+1}\times X} |f(g_0,\ldots,g_k,x)|\bigl(1+\diam(g_0,\ldots,g_k)\bigr)<\infty\Bigr\};
\end{multline*}
similarly for $C^{(1,1)}_k(G) \widehat{\otimes}_{\bbZ} L^1(X,\bbR)$.
\end{example}

\begin{remark}\label{rem:recover projective tensor product}
Let $E$ and $F$ be Banach spaces over $\bbR$.
Observe that
\[E\widehat{\otimes}_\bbZ F \cong E\widehat{\otimes}_\bbR F,\]
since the left hand side satisfies the universal property of the right hand side.
The universal property satisfied by  $E\widehat{\otimes}_\bbR F$
is the one satisfied by the classical projective tensor product of Banach spaces.
It follows that Definition~\ref{def:projective tensor product} generalizes the classical definition of projective tensor product (see \cites{grothendieck, diestel}).
\end{remark}

\section{Theorem~\ref{bounded-chian-map} implies Theorem~\ref{hyp-sob}}
\label{sec:proof of main thm}

Throughout this section, let $G$ be a finitely generated group with a
fixed word metric. We consider $\bbZ G$ as a normed ring
endowed with the $\ell^1$ norm.

\begin{definition}[Rips complex]\label{def: rips complex}
	Let $X$ be a metric space.
	Let $r>0$. We denote by $C^r_\ast(X)$ the subcomplex of $C_\ast(X)$ such that
	$C^r_n(X)\subset C_n(X)$ is the submodule generated by all $(n+1)$-tuples
	$(x_0,\dots,x_n)$ whose diameter in $X$ is at most $r$.
	If $X=G$ is a group as above, then
	$C_\ast^r(G)$ is a $\bbZ G$-subcomplex of $C_\ast(G)$.
\end{definition}

\begin{theorem}[\cite{bridson+haefliger}*{3.23~Proposition on p.~469}]\label{thm:rips}
	Let $G$ by a $\delta$-hyperbolic group. If $r\ge 4\delta+6$, then
    \[
    	H_n(C^r_\ast(G))=\begin{cases}
						\bbZ & \text{ if $n=0$,}\\
						0    & \text{ if $n>0$.}
					\end{cases}
    \]	
\end{theorem}

If we endow $C_\ast^r(G)$ with the
$\ell^1$-norm, we use sometimes the notation $C_\ast^{r,(1)}(G)$.

\begin{lemma}\label{lem:tensor with complete module}
Let $E$ be a complete normed $\bbZ G$-module. For every $n\in\bbN$ the natural
map
\begin{equation}\label{eq:natural map over Z}
C^r_n(G)\otimes_{\bbZ G}E \to C^{r,(1)}_n(G)\widehat\otimes_{\bbZ G}E
\end{equation}
is an isomorphism.
\end{lemma}

\begin{proof}
Let $N$ be the number of all tuples $(e,g_1,\ldots,g_n)$ in $G^{n+1}$ with
diameter at most $r$. The $\bbZ G$-module $C^r_n(G)=C^{r,(1)}_n(G)$ is isomorphic
to the free module $\bbZ G^N$. From that and from
Lemma~\ref{cor:tensor}
one sees that both sides
in~\eqref{eq:natural map over Z} are canonically isomorphic to $E^N$.
\end{proof}

We need the following continuous version of the fundamental lemma in
homological algebra.

\begin{lemma} \label{lem:basic homotopy}
Let $n\ge 1$.
Let $\phi_i\colon C_i^{(1,1)}(G)\to C_i^{(1)}(G)$, $0\le i\le n$, be a $\bbZ G$-chain homomorphism
up to degree $n$, that is, we have $d\phi_i=\phi_{i-1}d$ for every $0\le i\le n$.
Assume that $\phi$ induces the identity on the zeroth homology.
Then there are bounded $\bbZ G$-homomorphisms
$h_i\colon C_i^{(1,1)}(G)\to C_{i+1}^{(1)}(G)$, $0\le i\le n$, such that
\[
	dh_i+h_{i-1}d=\phi_i-\id~~\text{ for every $0\le i\le n$,}
\]
where $h_{-1}d=0$ is understood.
That is, $h_\ast$ is a chain homotopy between $\phi_\ast$ and the identity
up to degree $n$.
\end{lemma}

\begin{proof}
Recall that $C_\ast^{(1,1)}(G)=C_\ast(G)$ and $C_\ast^{(1)}(G)=C_\ast(G)$ as
$\bbZ G$-modules.
One verifies that $h'_i:C_i(G)\to C_{i+1}(G)$ defined
by
\[
	h'_i(g_0,\dots,g_i)=(e,g_0,\dots,g_i)
\]
is a (non-equivariant) chain contraction of the augmented chain
complex $C_\ast(G)$ (see the comment on augmented chain complexes after Lemma~
\ref{lem:tree-homotopy}).

Obviously, $h'_\ast$
is continuous with respect to the $\ell^1$-norm. Let $x\in C_1(G)$ be an element
such that $\phi_0(e)-e=dx$ for the unit $e\in G$;
this element exists since $\phi_\ast$ induces
the identity on $0$-th homology. Then let $h_0\colon C_0(G)\to C_1(G)$ be the $\bbZ G$-homomorphism
with $h_0(g)=g x$. Clearly, $h_0$ is bounded and satisfies
$\phi_0-\id=dh_0$.

Now suppose that we have already
constructed an equivariant bounded map $h_i:C_i(G)\to C_{i+1}(G)$ for
$i=0,\ldots,k-1$, where $k\le n$, such that
\begin{equation}\label{eq:homotopy}
	dh_i+h_{i-1}d=\phi_i-\id.
\end{equation}
for all $i=1,\ldots,k-1$ (where we set $h_{-1}=0$).
Then define
\[
	h_k(e,g_1,\ldots,g_k)=\bigl(h'_k\circ ( \phi_k-\id-h_{k-1}d) \bigr)(e,g_1,\ldots,g_k)
\]
and extend $h_k$ to all of $C_k(G)$ by $\bbZ G$-linearity. It is easy to see that $h_k$
is bounded with respect to the Sobolev norm in the domain and the $\ell^1$-norm in the
target and satisfies~\eqref{eq:homotopy}.
\end{proof}

\begin{proof}[Proof that Theorem~\ref{bounded-chian-map} implies Theorem~\ref{hyp-sob}]
Let $G$ be a $\delta$-hyperbolic group.
Let $E$ be a complete normed $\bbZ G$-module.
The $\subset$-inclusion in the statement of
Theorem~\ref{hyp-sob} is clear. Let $n\ge 0$. It remains to show that
\begin{equation}\label{eq:to show}
	\im\Bigl(H_n(G, E)\to H_n^{(1)}(G,E)\Bigr)\supset
	\im\Bigl(H_n^{(1,1)}(G,E)\to H_n^{(1)}(G,E)\Bigr).
\end{equation}

Let $r(i)=r(i,\delta)$, $i\ge 0$, be the constants and
$f_\ast\colon C_\ast(G)\to C_\ast(G)$ be the map
provided by
Theorem~\ref{bounded-chian-map}. Let
\[
	r=\max\{4\delta+6, r(0),r(1),\ldots, r(n+1)\}.
\]
The complex $C^r_\ast(G)$ is acyclic according to Theorem~\ref{thm:rips}.
We have $\im(f_i)\subset C^r_i(G)$ for every $0\le i\le n+1$.

The map $f_\ast$ is
a bounded chain homomorphism $f_\ast\colon C_\ast^{(1,1)}(G)\to C^{r,(1)}_\ast(G)$
up to degree $n+1$.
Since $C^r_\ast(G)$ is acyclic, it is a free $\bbZ G$-resolution of $\bbZ$. Thus,
by the fundamental lemma of homological algebra, there is a $\bbZ G$-chain map
$g_\ast\colon C^r_\ast(G)\to C_\ast(G)$. Since $C_i^r(G)$ is finitely generated as a $\bbZ G$-module for each $i\ge 0$, the map $g_i$ is automatically continuous with
respect to the $\ell^1$-norms. Consider the following diagram for $\ast\le n+1$:

\[
	\xymatrix{C_\ast^{(1,1)}(G)\widehat\otimes_{\bbZ G} E\ar[r]^{f_\ast\widehat\otimes\id_E}\ar[d] & C_\ast^{r,(1)}(G)\widehat\otimes_{\bbZ G} E &C_\ast^{r,(1)}(G)\otimes_{\bbZ G} E\ar[d]^{g_\ast\otimes\id_E}\ar@{=}[l]\\
	C_\ast^{(1)}(G)\widehat\otimes_{\bbZ G} E & &C_\ast(G)\otimes_{\bbZ G} E\ar[ll]}.
\]
Note that $f_\ast$ induces a map $f_\ast\widehat\otimes\id_E$ on the completed tensor
products since it is continuous.
The unlabeled arrows in the diagram are induced by natural inclusions.
The equality in the diagram follows from Lemma~\ref{lem:tensor with complete module}.
The diagram is commutative up to chain homotopy by Lemma~\ref{lem:basic homotopy}. This
implies~\eqref{eq:to show}.
\end{proof}

\section{Tree approximation and the proof of Theorem~\ref{bounded-chian-map}} % (fold)
\label{sec:tree_approximation_and_the_proof_of_theorem_bounded-chian-map}

\subsection{Tree approximation} % (fold)
\label{sub:tree_approximation}

% subsection tree_approximation (end)

\begin{definition}
	Let $G$ by a finitely generated group. Fix a finite symmetric generating set in $G$.
	Let $\calG$ be the corresponding Cayley graph of $G$.
	A family $\calW=\{w_{x,y}\}_{(x,y)\in G^2}$ is called
	a \emph{full family of geodesics in $G$} if
	$w_{x,y}\colon [0,d(x,y)]\to\calG$ is a geodesic from $x$ to $y$
	in $\calG$ for every pair $(x,y)\in G^2$. For any
	$n$-tuple $Y=(y_0,\ldots,y_{n-1})\in G^n$
	let $[Y]_\calW\subset G$ be the set of vertices of the union $\calG(Y)$
	of the images of all geodesics $w_{y_i,y_j}$ with
	$i<j$.
\end{definition}

Recall that a map between metric spaces $f:X\to Y$ is called a \emph{$c$-rough isometry}
if for every $x,x'\in X$ we have
\[ \abs{d(x,x')-d(f(x),f(x'))}\le c\text{ for all $x,x'\in X$}. \]
The metric spaces $X$ and $Y$ are \emph{$c$-roughly isometric} if
there are $c$-rough isometries $f:X\to Y$, $g:Y\to X$ such that
$d(x,g(f(x)))\le c$ and $d(y,f(g(y)))\le c$ for every $x\in X$ and every $y\in Y$.
Furthermore, if there is a $c$-rough isometry $f:X\to Y$ such that $f(X)$ is $c$-dense
in $Y$, then $X$ and $Y$ are $2c$-roughly isometric.

A \emph{metric simplicial tree} $T=(V,E)$ is a simplicial tree
with vertices $V$ and edges $E$ endowed with a path
metric $d$ on (the geometric realization of) $T$
such that each edge $e\in E$ is isometric to a compact interval $[0,l_e]\subset\bbR$.

\begin{theorem}\label{thm:rough isometry to a tree}
	Let $G$ be a $\delta$-hyperbolic group. Let
	$\calW=\{w_{x,y}\}_{(x,y)\in G^2}$ be a full family of geodesics in $G$.
	For every $n\in\bbN$ there is a constant $c=c(\delta,n)>0$ such that for
	every $n$-tuple $Y\in G^n$ the subspace $[Y]_\calW$
	is $c$-roughly isometric to a metric simplicial tree.
\end{theorem}

\begin{proof}
	Consider a $n$-tuple $Y=(y_0,\ldots,y_{n-1})\subset G$.
	Let $\calG$ be the Cayley graph of $G$. Let
	\begin{align*}
			\calG(Y)_0&=\bigcup_{0<i<n}\im(w_{y_0,y_i}),\\
			\calG(Y)  &=\bigcup_{0\le i<j<n}\im(w_{y_i,y_j}).
	\end{align*}
	Since the set of vertices in $\calG(Y)$, which is just $[Y]=[Y]_\calW$,
	is $1$-dense,
	if $\calG(Y)$ is $c$-roughly
	isometric to some metric space, then $[Y]$ is $(c+2)$-roughly isometric to the same
	space. We will now construct a rough isometry of $\calG(Y)$ to a metric simplicial tree.
	
	It is proved
	in~\cite{delzant}*{Th\'eor\`eme~1 on p.~91} that there is a constant
	$c'(\delta, n)>0$, which only depends on $\delta$ and $n$,
	and a map $f:\calG(Y)_0\to T$ to a metric simplicial tree $T$
	such that for each $0<i<n$ the restriction $f\vert_{\im(w_{y_0,y_i})}$ is a bijective
	isometry and
	\begin{equation}\label{eq:rough inequality}
		d(x,y)-c'(\delta,n)\le d(f(x),f(y))\le d(x,y)
	\end{equation}
	for all $x,y\in \calG(Y)_0$. Note that $f$ is automatically surjective.
	The unique geodesic segment between points $z$ and
	$z'$ of $T$ will be denoted by $[z,z']\subset T$.
	Next we extend $f$ to $\calG(Y)$ as follows: For every
	$x\in \calG(Y)\bs \calG(Y)_0$ choose $0<a(x)\le n-1$ and $0<b(x)\le n-1$ such that
	$x\in\im(w_{y_{a(x)},y_{b(x)}})$. Because of~\eqref{eq:rough inequality}
	we can pick a point $z\in [f(y_{a(x)}), f(y_{b(x)})]$ such
	that
	\begin{gather*}
		\abs{d\bigl(z, f(y_{a(x)})\bigr)-d\bigl(x, y_{a(x)}\bigr)}<c'(\delta, n),\\
		\abs{d\bigl(z, f(y_{b(x)})\bigr)-d\bigl(x, y_{b(x)}\bigr)}<c'(\delta, n).
	\end{gather*}
	Then set $f(x)=z$.

	To finish the proof, we show that
	$f:\calG(Y)\to T$ satisfies
	\begin{equation}\label{eq:rough final}
		\forall x,x'\in \calG(Y)\colon~ \abs{d(x,x')-d(f(x),f(x'))}\le c
	\end{equation}
	for
	\[
		c=4\delta+3c'(\delta,n).
	\]
	Let $x$ and $x'$ be points on the geodesics $w_{y_i,y_j}$
	and $w_{y_{i'},y_{j'}}$, where $i=a(x), j=b(x)$ and $i'=a(x'), j'=b(x')$.
	By $\delta$-hyperbolicity
	there is a point
	$z$ on $w_{y_0,y_i}$ or $w_{y_0,y_j}$ such that
	\begin{equation}\label{eq:distance since hyperbolic}
	d(x,z)<\delta.
	\end{equation}
	By the triangle inequality we obtain that
	\[
		\abs{d(y_i,x)-d(y_i,z)}<\delta~\text{ and }~
		\abs{d(y_j,x)-d(y_j,z)}<\delta.
	\]
	Thus,
	\begin{align}\label{eq:rough distance}
		\abs{d(f(y_i),f(x))-d(f(y_i),f(z))}&<\delta+c'(\delta,n),\\
		\abs{d(f(y_j),f(x))-d(f(y_j),f(z))}&<\delta+c'(\delta,n).\notag
	\end{align}
	Since $T$ is tree and $f(x)\in [f(y_i),f(y_j)]$, either $f(x)\in [f(y_i),f(z)]$, or
	$f(x)\in [f(z), f(y_j)]$. In both cases~\eqref{eq:rough distance} implies that
	\begin{equation}\label{eq:f distance 1}
		d(f(x),f(z))<\delta+c'(\delta,n).
	\end{equation}
	Similarly, we find a point $z'$ on $w_{y_0,y_{i'}}$ or $w_{y_0,y_{j'}}$ such that
	\begin{equation}\label{eq:f distance 2}
		d(x',z')<\delta~\text{ and }~d(f(x'),f(z'))<\delta+c'(\delta,n).
	\end{equation}
	From~\eqref{eq:f distance 1},~\eqref{eq:f distance 2},~\eqref{eq:distance since hyperbolic}
	and the fact that $f$ is a $c'(\delta,n)$-rough isometry on $\calG(Y)_0$ we obtain that
	\begin{align*}
		d(x,x') \le d(z,z')+2\delta &\le d(f(z),f(z'))+2\delta+c'(\delta,n)\\
			&\le d(f(x),f(x'))+4\delta+3c'(\delta,n).
	\end{align*}
	Similary, we get
	\[
		d(x,x') \ge d(f(x),f(x'))+4\delta+3c'(\delta,n),
	\]
    which proves~\eqref{eq:rough final}.
\end{proof}

\subsection{An efficient chain contraction of the Rips complex to a tree} % (fold)
\label{sub:an_efficient_chain_contraction_of_the_rips_complex_of_a_tree}

% subsection an_efficient_chain_contraction_of_the_rips_complex_of_a_tree (end)

\begin{lemma}\label{lem:tree-homotopy}
For every $r\ge 1$ and every $n\in\bbN$ there is a constant $e(r,n)>0$
with the following property:
Let $T$ be a metric simplicial tree. Let $V$ be a subset of the vertices of $T$ such
that the distance of any two distinct
vertices in $V$ is at least $1$.
Then there is a chain contraction $h_i^T:C^r_i(V)\to C^r_{i+1}(V)$, $i\ge -1$,
of the augmented chain complex $C^r_\ast(V)$
such that
\begin{equation}\label{eq:operator norm}
\norm{h_i^T}<e(r,i)~\text{ for every $i\ge 1$},
\end{equation}
where the operator norm is taken with respect to the $\ell^1$-norms.
\end{lemma}

Here we mean by the \emph{augmented} chain complex the complex $C_\ast^r(V)$
extended by $C_{-1}^r(V)=\bbZ$ and
the differential (\emph{augmentation}) $d:C_0^r(V)\to\bbZ$ that maps every $v\in V$
to $1\in\bbZ$.

\begin{proof}
	Fix a base point $x\in V$. Let $h^T_{-1}$ be defined by $h^T_{-1}(1)=x$. For every
	$v\in V$ we define $h^T_0$ by
	\[
		h^T_0(v)=\sum_{i=0}^{m-1} (v_i, v_{i+1}),
	\]
	where $x=v_0,v_1,\ldots,v_m=v$ (in that order) are the vertices in $V$ lying
	on the unique geodesic
	from $x$ to $v$. It is clear that
	\[
		dh_0^T(v)=v-x=(\id-h_{-1}^Td)(x).
	\]
	For $v\in V$ and $i\ge 0$ consider the linear map given by
	\[
		c_v\colon C_i(V)\to C_{i+1}(V),~c_v(v_0,\ldots, v_i)=(v,v_0,\ldots, v_i).
	\]
	One verifies that for $i\ge 1$, $v\in V$, and $(v_0,\ldots,v_i)\in V^{i+1}$ we have
	\[
		dc_v(v_0,\ldots,v_i)=(\id - c_v d)(v_0,\ldots,v_i).
	\]
	We define the homomorphisms
	$h^T_i\colon C_i^r(V)\to C_{i+1}(V)$ for $i=1,2,\ldots$ inductively by
	\begin{equation}\label{eq:inductive definition of homotopy}
		h^T_i(v_0,\ldots, v_i)=c_{v_0}\bigl((\id-h^T_{i-1}d)(v_0,\ldots, v_i)\bigr).
	\end{equation}
	It follows inductively from the following computation that
	$h^T_\ast$ is a chain contraction:
	\begin{align*}
		dh^T_{i+1}(v_0,\ldots,v_{i+1})&= dc_{v_0}(\id-h^T_i d)(v_0,\ldots,v_{i+1})\\
			&= (\id-c_{v_0}d)(\id-h^T_i d)(v_0,\ldots,v_{i+1})\\
			&= (\id-h^T_id-c_{v_0}d+c_{v_0}dh^T_id)(v_0,\ldots,v_{i+1})\\
			&= (\id-h^T_id-c_{v_0}d+c_{v_0}(\id-h^T_{i-1}d)d)(v_0,\ldots,v_{i+1})\\
			&= (\id-h^T_id)(v_0,\ldots, v_{i+1}).
	\end{align*}
	Next we define $e(r,i)$ and show that~\eqref{eq:operator norm} holds by induction.
	Set $e(r,1)=r+1$. Let $(u,v)\in C_1^r(V)$ be a $1$-simplex. Let
	$v=z_0,z_1,\ldots, z_m=u$ (in that order) be the vertices in $V$ lying
	on the unique geodesic from $v$ to $u$.
	Since $T$ is a tree, we get that
	\begin{equation}\label{eq:h_0d}
		h_0^T d(u,v)=\sum_{k=0}^{m-1} (z_k,z_{k+1}).
	\end{equation}
	Since the distance from $u$ to $v$ is $\le r$ and the distance
	from $z_k$ to $z_{k+1}$ is $\ge 1$ by assumption,
	we have $m\le r$ and thus
	$\norm{h_0d(u,v)}_1\le r$. This implies $\norm{h_1}\le e(r,1)$.
	For $i\ge 2$ set
	\[
		e(r,i)=e(r,i-1)\cdot (i+1)+1.
	\]
	Because of $e(r,1)=r+1$ one sees that $e(r,i)$ only depends on $r$ and $i$, but
	not on the specific tree $T$. Definition~\eqref{eq:inductive definition of homotopy}
	and the fact that the differential in degree $i$ has norm at most $i+1$
	yield~\eqref{eq:operator norm}.
	
	Finally we prove that $\im(h_i^T)\subset C_{i+1}^r(V)$.
	It suffices to show that
	for every $i\ge 1$ and every $(v_0,\ldots, v_i)\in C_i^r(V)$ we have
	\begin{equation}\label{eq:support}
		\supp\bigl(h_i^T(v_0,\ldots,v_i)\bigr)\subset\conv(v_0,\ldots,v_i).
	\end{equation}
	Here the \emph{support} $\supp(s)$ of an element $s\in C_i(V)$, which can be uniquely
	written as a linear combination of $(i+1)$-tuples in $V^{i+1}$, is the union of all
	$v\in V$ that appear in one of these $(i+1)$-tuples. We denote the \emph{convex hull}
	of a set $S\subset V$ by $\conv(S)$.
	We have~\eqref{eq:support} for $i=1$ by
	definition~\eqref{eq:inductive definition of homotopy} and because all the points
	$z_i$ in~\eqref{eq:h_0d} on the geodesic from $u$ to $v$.
	If~\eqref{eq:support} holds for $h_i^T$ with $i\ge 1$, it is true for $h_{i+1}^T$
	because of:
	\begin{align*}
		\supp \bigl(h_{i+1}^T(v_0,\ldots, v_{i+1})\bigr) &
		    \subset \{v_0,\ldots,v_{i+1}\}\cup\bigcup_{k=0}^{i+1}\supp\bigl( h_i^T(v_0,\ldots,\widehat v_k,\ldots,v_{i+1})\bigr)\\
		&\subset \{v_0,\ldots,v_{i+1}\}\cup\bigcup_{k=0}^{i+1}\conv(v_0,\ldots,\widehat v_k,\ldots,v_{i+1})\\
		&=\conv(v_0,\ldots,v_{i+1}).\qedhere
	\end{align*}
\end{proof}

\subsection{Proof of Theorem~\ref{bounded-chian-map}} % (fold)
\label{sub:proof_of_theorem_bound}

\begin{proof}
	Choose a full family $\calW=\{w_{x,y}\}_{(x,y)\in G^2}$
	of geodesics in $G$ that is \mbox{$G$-equi}\-variant in the sense that
	for all $x,y,g\in G$ we have $gw_{x,y}=w_{gx,gy}$.
		
	For a $k$-tuple $Y\in G^k$ we write
	$[Y]$ instead of $[Y]_\calW$ in the sequel.
	For $i=0,1,\ldots$ we
	define inductively real numbers $r(i)\ge 1$ and $\bbZ G$-homomorphisms
	$f_i\colon C_i(G)\to C^{r(i)}_i(G)$ for $i\in\bbN$ such that
	\begin{enumerate}[label=\alph*)]
	\item $f_0$ is the identity,
	\item $d f_i=f_{i-1} d$,
	\item $f_i$ is bounded when endowing the source with the Sobolev and the target
	with the $\ell^1$ norm, and
	\item for every $(g_0,\ldots,g_i)\in C_i(G)$ we have
	\[f_i\bigl((g_0,\ldots,g_i)\bigr)\in C^{r(i)}_i\bigl([(g_0,\ldots,g_i)]\bigr).\]
	\end{enumerate}

    The theorem follows from a)-c).
	Property d) is just needed for running the induction argument.

	The basis of the induction will be an explicit construction of $f_0$ and $f_1$.
	We set $r(0)=r(1)=1$. Define $f_0$ to be the identity map.
    If for $x,y\in G$ the points $x=z_0,z_1,\ldots,z_d=y$ are the successive vertices on the
	geodesic $w_{x,y}$ from $x$ to $y$, we
	define $f_1$ by
	\[ f_1\bigl((x,y)\bigr)=\begin{cases} \sum_{i=1}^d (z_{i-1},z_i) & \text{ if $i\ge 1$,}\\
	                          (x,y) & \text{ if $i=0$ and~$x=y$}.
				\end{cases}\]
	It is clear that $f_0$ and $f_1$ respect all the properties above.

	Fix $i\ge 1$ and assume $f_j$ is already defined for $0\le j\le i$
	satisfying a)--d).
	According to Theorem~\ref{thm:rough isometry to a tree} there is a constant
	$c(i)>0$ such that for every $Y\in G^{i+1}$ the
	subspace $[Y]$ is $c(i)$-roughly isometric to metric simplicial tree.
	We set
	\[
		r(i+1)=r(i)+2c(i).
	\]
	By Lemma~\ref{lem:tree-homotopy}
	for every $n\in\bbN$ there is a constant $e(i,n)>0$ such that for every
	metric simplicial tree $T$ with a
	subset $V$ of vertices whose distinct elements have distance $\ge 1$
	from each other
	there is a chain contraction
	\[h_\ast^T\colon C_\ast^{r(i)+c(i)}(V)\to C_{\ast+1}^{r(i)+c(i)}(V)\]
	of the augmented chain complex
	such that
	the operator norm with respect to the $\ell^1$-norm satisfies
	$\norm{h_n}<e(i,n)$. Let
	\[B=\bigl\{(e,g_1,\ldots,g_{i+1});~\text{$g_k\in G$ for $1\le k\le i+1$}\bigr\}.\]
	Note that $B$ is $\bbZ G$-basis of $C_{i+1}(G)$.
	After some preparation
	we define $f_{i+1}(\sigma)\in C_{i+1}^{r(i+1)}(G)$ in~\eqref{eq:defining f}
	for every $\sigma\in B$
	such that
	\begin{enumerate}[label=\alph*), start=5]
		\item $df_{i+1}(\sigma)=f_id(\sigma)$,
		\item $\norm{f_{i+1}(\sigma)}_1\le \bigl(e(i,i+1)+(i+1)\bigr)\norm{f_i}(i+2)\norm{\sigma}_{1,1}$, and
		\item $f_{i+1}(\sigma)\in C_{i+1}^{r(i+1)}\bigl([\sigma]\bigr)$
	\end{enumerate}
	hold for every $\sigma\in B$. The theorem then follows from the following easy claim,
	which we leave to the reader.
	
	\emph{Claim.} The $\bbZ G$-linear extension
	to $C_{i+1}(G)\to C_{i+1}^{r(i+1)}(G)$ of a map
	$f_{i+1}\colon B\to C_{i+1}^{r(i+1)}(G)$ satisfying e)--g)
	satisfies b)--d). The extension
	$f_{i+1}\colon C_{i+1}(G)\to C_{i+1}^{r(i+1)}(G)$ has
	operator norm
	\[
		\norm{f_{i+1}}\le  \bigl(e(i,i+1)+(i+1)\bigr)\norm{f_i}(i+2).
	\]
	\indent Let $\sigma\in B$. Let $T^\sigma$ be a metric
	simplicial tree such that
	$[\sigma]$ is $c(i)$-roughly isometric to $T^\sigma$. Let $V^\sigma$
	be a set of points of $T^\sigma$ such that any two distinct
	points in $V^\sigma$ have distance $\ge 1$ and $V^\sigma$ is $3$-dense
	in $T^\sigma$. By subdividing $T^\sigma$ we may assume that $V^\sigma$
	consists of vertices.
	Upon increasing
	$c(i)$ by $6=2\cdot 3$, thus $r(i+1)$ by $12=2\cdot 6$, we may and will assume that
    $[\sigma]$ is also $c(i)$-roughly isometric to $V^\sigma$.
	Let
	\begin{equation*}
	\phi^\sigma\colon [\sigma]\to V^\sigma~\text{ and }~
	\psi^\sigma\colon V^\sigma\to [\sigma]
	\end{equation*}
	two $c(i)$-rough isometries such that
	\[d(\phi^\sigma \psi^\sigma,\id_{V^\sigma})\le c(i)~\text{ and }~
	  d(\psi^\sigma \phi^\sigma,\id_{[\sigma]})\le c(i).
	\]
	The maps $\phi^\sigma$ and $\psi^\sigma$ induce chain maps
	\[\phi_\ast^\sigma\colon C_\ast^{r(i)}\bigl([\sigma]\bigr)
	\to C_\ast^{r(i)+c(i)}(V^\sigma)
	~\text{ and }~
	\psi_\ast^\sigma\colon C_\ast^{r(i)+c(i)}(V^\sigma)\to
	C_\ast^{r(i+1)}\bigl([\sigma]\bigr).
	\]
	The following claim follows by a straightforward computation.
	
	\emph{Claim.} Let $r>0$.
		The map $h_\ast^{\sigma}\colon
		C_\ast^{r(i)}\bigl ([\sigma]\bigr)\to
		C_{\ast+1}^{r(i+1)}([\sigma])$ defined by
		\[
			h_n^\sigma(g_0,\ldots,g_n)=\sum_{k=0}^n(-1)^k\bigl(g_0,\ldots,g_k,\psi^\sigma\phi^\sigma(g_k),\ldots, \psi^\sigma\phi^\sigma(g_n)\bigr)
		\]
		is a chain homotopy between the composition
		$\psi^\sigma_\ast\phi^\sigma_\ast\colon C_\ast^{r(i)}([\sigma])
		\to C_\ast^{r(i+1)}([\sigma])$ and the identity, that is,
		$\psi^\sigma_n\phi^\sigma_n-\id=dh_n^\sigma+h_{n-1}^\sigma d$
		for every $n\ge 0$, where we set $h_{-1}^\sigma=0$.
				
	For $\sigma\in B$ define now
	\begin{equation}\label{eq:defining f}
		f_{i+1}(\sigma)=\psi_{i+1}^\sigma h_i^{T_\sigma}\phi^\sigma_if_i d(\sigma)-
		h_i^\sigma f_i d(\sigma)\in C_{i+1}^{r(i+1)}([\sigma]).
	\end{equation}
	Property g) is clear from the definitions. The differential in degree $(i+1)$
	of $C_\ast(G)$ (endowed with the Sobolev norm) has norm at most $(i+2)$,
	and $h_i^\sigma$ has norm at most $(i+1)$ (with respect to the $\ell^1$-norms).
	The maps $\phi_i^\sigma$ and $\psi_i^\sigma$ have norm at most $1$. Hence
    we obtain that
	\[
		\norm{f_{i+1}(\sigma)}_1\le \bigl(e(i,i+1)+(i+1)\bigr)\norm{f_i}(i+2)\norm{\sigma}_{1,1}.
	\]
	Property e) follows from:
	\begin{align*}
		df_{i+1}(\sigma) &=
		d\bigl( \psi_{i+1}^\sigma h_i^{T_\sigma}\phi_i^\sigma f_id-h_i^\sigma f_i d\bigr)(\sigma)\\
		&= \bigl(\psi_i^\sigma d h_i^{T_\sigma}\phi_i^\sigma f_id-dh_i^\sigma f_i d\bigr)(\sigma)\\
		&= \bigl( \psi_i^\sigma(\id-h_{i-1}^{T_\sigma}d)\phi_i^\sigma f_i d-(\psi_i^\sigma\phi_i^\sigma-\id-h_{i-1}^\sigma d)f_i d\bigr)(\sigma)\\
		&=\bigl(f_id -\psi_i^\sigma h_{i-1}^{T_\sigma} d\phi_i^\sigma f_id\bigr)(\sigma)+\bigl(h_{i-1}^\sigma f_{i-1} dd\bigr)(\sigma)\\
		&=\bigl(f_id -\psi_i^\sigma h_{i-1}^{T_\sigma} \phi_{i-1}^\sigma f_{i-1} dd\bigr)(\sigma)\\
		&=f_i d(\sigma).\qedhere
	\end{align*}
\end{proof}

\section{Integrable measure equivalence and simplicial volume}
\label{sec:_ell_1_measure_equivalence_and_simplicial_volume}

\subsection{Integrable measure equivalence} % (fold)
\label{sub:_ell_1_measure_equivalence}

We recall the central notion of \emph{measure equivalence} which
was suggested by Gromov~\cite{gromov-invariants}*{0.5.E}.

\begin{definition}
	Two countable groups ${G}$, ${H}$ are called \emph{measure equivalent}
	if there is a standard measure space $(\Omega, \mu)$ with commuting
	$\mu$-preserving ${G}$- and ${H}$-actions,
	such that each one of the actions admits
	a finite measure fundamental domain.
	The space $(\Omega, \mu)$ endowed with these actions is called an \emph{\mbox{ME-coupling}}
	of ${G}$ and ${H}$.
\end{definition}

Given measure equivalent groups $G$ and $H$, an actual choice of fundamental domains is not a part of the structure
of an ME-coupling of $G$ and $H$.
But it is easy to see that the \emph{measures} of $G$- and
$H$-fundamental domains are independent of the choice.
So for an ME-coupling $(\Omega, \mu)$ of $G$ and $H$,
the ratio $c_{\Omega}=\mu(X_H)/\mu(X_G)$ of the measure of an $H$-fundamental domain by the
measure of a $G$-fundamental domain is well defined and called
the {\em coupling index} of $\Omega$.

The map $X_G\hookrightarrow\Omega\twoheadrightarrow G\bs\Omega$ is a measure isomorphism.
Since $H$ acts on $G\bs\Omega$, this identification induces a measurable action of $H$
on $X_G$, for which we use the dot notation $h\cdot x$ for $h\in H$ and $x\in X_H$ to
distinguish it from the action $hx$ of $H$ on $\Omega$. Similarly for $X_H$.

The coupling $\Omega$ is called
\emph{ergodic} if the $H$-action on $G\bs\Omega$ is ergodic,
or equivalently, the $G$-action on $H\bs\Omega$ is ergodic~\cite{furman}*{Lemma~2.2}.

\begin{definition}\label{def:coycles}
	Let $(\Omega,\mu)$ be an ME-coupling of $G$ and $H$. Let $X_G\subset\Omega$ and
	$X_H\subset\Omega$ be fundamental domains of the $G$- and $H$-action, respectively.
	\begin{enumerate}
		\item We define $\alpha_{X_H}$ as
		\[
			\alpha_{X_{H}}\colon{G}\times X_{H}\to{H},~({g},x)\mapsto {h}
			\text{ with ${g} x\in{h}^{-1} X_{H}$},
		\]
		and call $\alpha_{X_H}$ the \emph{(measurable) cocycle associated to $X_H$}.
		Similarly for $\alpha_{X_G}$.
		\item Assume that $H$ is finitely generated, and let $l\colon{H}\to\bbN$ be the
		length function associated to some word-metric on ${H}$. We
		say that the fundamental domain $X_H$ is \emph{integrable}
		if the function $x\mapsto l(\alpha_{X_H}(g,x))$ is in $L^1(X_H)$
		for every $g\in G$. Similarly for $X_G$.
	\end{enumerate}
\end{definition}

\begin{definition}\label{def:l1-measure coupling}
	Let ${G}$ and ${H}$ be finitely generated.
	We say that an ME-coupling of ${G}$ and ${H}$ is an \emph{$\ell^1$-ME-coupling} or
	an \emph{integrable ME-coupling}
	if it admits integrable ${G}$- and ${H}$-fundamental domains. We say that
	$G$ and $H$ are \emph{$\ell^1$-measure equivalent} if there exists an $\ell^1$-ME-coupling
	of $G$ and $H$.
\end{definition}

\begin{remark}
	Measure equivalence and $\ell^1$-measure equivalence are equivalence
	relations on countable and finitely generated groups, respectively
	(see~\cites{furman, mostow+me}).
\end{remark}

\begin{remark}
A locally compact group $G$ with its Haar measure is an ME-coupling for all its
lattices; it is an integrable ME-coupling for
every pair of cocompact lattices in $G$.
	
	By~\cite{shalom}*{Theorem~3.6} the isometry group $\operatorname{Isom}(\bbH^n)$ of the $n$-dimensional hyperbolic space with $n\ge 3$ endowed with its Haar measure is
	an $\ell^1$-ME-coupling for all its lattices.
	Shalom in \emph{loc.~cit.} was concerned with $\ell^2$-integrability
	and showed  that all lattices in simple Lie groups
	not locally
	isomorphic to $\isom(\bbH^2)\cong PSL_2(\bbR)$ or
	$\isom(\bbH^3)\cong PSL_2(\bbC)$
	are \emph{$\ell^2$-integrable}.
	However, his proof also implies the above statement.
\end{remark}

\subsection{Bounded cohomology and ME-induction}
\label{sub:bounded_cohomology_and_me_induction}

We briefly recollect basic notions of \emph{bounded cohomology}.

Let ${G}$ be a discrete group and $E$ be a real Banach space with
isometric ${G}$-action.
We denote by $\rmCb^k({G},E)$ the
Banach space $L^\infty({G}^{k+1},E)$ consisting of
bounded maps from ${G}^{k+1}$ to $E$ endowed with the supremum norm
and the isometric ${G}$-action:
\begin{equation*}
  ({g}\cdot f)({g}_0,\dots,{g}_k)={g}\cdot f({g}^{-1}{g}_0,\dots,{g}^{-1}{g}_k).
\end{equation*}
The sequence of Banach ${G}$-modules $\rmCb^k({G}, E)$, $k\ge
0$, becomes a chain complex of Banach ${G}$-modules via the
standard homogeneous coboundary operator
\begin{equation*}\label{eq:homogeneous differential}
  d(f)({g}_0,\dots,{g}_k)=\sum_{i\ge 0}^k(-1)^i
  f({g}_0,\dots, \hat{{g}_i},\dots,{g}_k).
\end{equation*}
The \emph{bounded cohomology} $\rmHb^\ast({G},E)$ of ${G}$ with
coefficients $E$ is the cohomology of the complex of
${G}$-invariants $\rmCb^\ast({G}, E)^G$.  The bounded cohomology
$\rmHb^\ast({G},E)$ inherits a semi-norm from $\rmCb^\ast({G},
E)$: The \emph{(semi-)norm} of an element $x\in \rmHb^k({G}, E)$ is
the infimum of the norms of all cocycles in the cohomology class $x$.

The topological dual of the complex of Banach spaces
\[C_\ast^{(1)}({G})\widehat{\otimes}_{\bbZ G}E\cong
C_\ast^{(1)}({G},\bbR)\widehat{\otimes}_{\bbR G} E,\]
whose homology $\rmHl_\ast({G},E)$ is the so-called \emph{$\ell^1$-homology
of $G$ with coefficents $E$}
(compare Remark~\ref{rem:recover projective tensor product}),
is canonically isomorphic to
$\rmCb^\ast({G}, E')$ (see~\cite{monod-book}*{Prop.~2.3.1 on p.~20}).
Thus, we obtain a natural pairing, which
descends to (co-)homology (both pairings are denoted by $\langle\_,\_\rangle$):
\begin{equation*}\label{eq:pairing on homology level}
  \langle\_,\_\rangle\colon: \rmHb^k({G},E')\otimes\rmHl_k({G},E)\to\bbR.
\end{equation*}

In the next theorem we identify the set $Hx\cap X_H$ which consists of just one element with the element itself.

\begin{theorem}[Monod-Shalom]\label{thm:induction from Monod-Shalom}
	Let $(\Omega,\mu)$ be a ME-coupling of ${G}$ and ${H}$. Let
	$X_{G}$ and $X_{H}$ be measurable fundamental domains for the ${G}$- and
	${H}$-action on $\Omega$, respectively.
	Let $\alpha\colon {H}\times X_{G}\to{G}$ be the cocycle associated
	to $X_{G}$.
	The maps 	
	\begin{gather}\label{eq:induction map on invariants}
		\alpha^\ast\colon\rmCb^\ast({G},\rmL^\infty(X_{H},\bbR))\to\rmCb^\ast({H},\rmL^\infty(X_{G},\bbR))\\
		\alpha^k f({h}_0,\dots,{h}_k)(x)=		f\bigl(\alpha({h}_0^{-1},x)^{-1},\dots,\alpha({h}_k^{-1},x)^{-1}\bigr)({H} x\cap X_{H})\notag
	\end{gather}
	define
	a chain map that restricts to the invariants
	\begin{equation*}
		\alpha^\ast\colon\rmCb^\ast({G},\rmL^\infty(X_{H},\bbR))^{G}
		\to
		\rmCb^\ast({H},\rmL^\infty(X_{G},\bbR))^{H}
	\end{equation*}
	and induces an isometric isomorphism
	\[\rmHb^\ast(\alpha)\colon \rmHb^\ast({G},\rmL^\infty(X_{H},\bbR))
	\xrightarrow{\cong}\rmHb^\ast({H},\rmL^\infty(X_{G},\bbR)).\]
	in cohomology. The map $\rmHb^\ast(\Omega)$ given by the commutative diagram
	\[\xymatrix{
		\rmHb^\ast({G},\rmL^\infty({H}\bs\Omega,\bbR))\ar[d]^\cong\ar[r]^{\rmHb^\ast(\Omega)} &
			\rmHb^\ast({H},\rmL^\infty({G}\bs\Omega,\bbR))\\
		\rmHb^\ast({G},\rmL^\infty(X_{H},\bbR))
		\ar[r]^{\rmHb^\ast(\alpha)}& \rmHb^\ast({H},\rmL^\infty(X_{G},\bbR))\ar[u]^\cong}
	\]
	where the vertical isomorphisms are induced by the (restrictions) of the
	projections $X_{H}\to{H}\bs\Omega$ and
	$X_{G}\to{G}\bs\Omega$, respectively, does not depend
	on the choices of fundamental domains.
\end{theorem}

\begin{proof}
	Apart from the fact that the isomorphism is isometric, this is exactly
	Proposition~4.6 in~\cite{monod+shalom} (with $S=\Omega$ and $E=\bbR$). The proof therein
	relies on~\cite{monod-book}*{Theorem~7.5.3 in~\S 7}, which also yields the
	isometry statement.
\end{proof}

To formulate the next theorem, consider the
measurable and countable-to-one map
\begin{gather*}
	\phi^\alpha_k\colon
	{H}^{k+1}\times X_{G}\to
	{G}^{k+1}\times X_{H}\\
	({h}_0,\dots,{h}_k,x)\to
	\bigl(\alpha({h}_0^{-1},x)^{-1},\dots,\alpha({h}_k^{-1},x)^{-1},
	{H} x\cap X_{H}\bigr).
\end{gather*}

\begin{theorem}\label{thm:induction in homology}
	Retain the notation of Theorem~\ref{thm:induction from Monod-Shalom}.
	Let $c_\Omega=\mu(X_{H})/\mu(X_{G})$ be the coupling index.
	We equip $X_{G}$ and $X_{H}$ with the
	normalized measures $\mu(X_{G})^{-1}\mu\vert_{X_{G}}$ and
	$\mu(X_{H})^{-1}\mu\vert_{X_{H}}$. Then
	\begin{gather}\label{eq:induction map on coinvariants}
		\alpha_k\colon C^{(1)}_k(H)\widehat\otimes_{\bbZ}\rmL^1(X_{G},\bbR)\to
		C^{(1)}_k(G)\widehat\otimes_\bbZ\rmL^1(X_{H},\bbR)\\
		\alpha_kf(\bar{g},x)=
		c_\Omega\cdot\sum_{(\bar{h},y)
		\in(\phi_k^\alpha)^{-1}(\bar{g},x)}
		f(\bar{h},y)\notag
		\end{gather}
	defines, using the identification in Example~\ref{ex:L1-C},
	a chain map
	that descends to the coinvariants
	\begin{equation*}
		C^{(1)}_\ast(H)\widehat\otimes_{\bbZ H}\rmL^1(X_{G},\bbR)\to
		C^{(1)}_\ast(G)\widehat\otimes_{\bbZ G}\rmL^1(X_{H},\bbR)
	\end{equation*}
	and induces an isometric isomorphism
	\[
		\rmHl_\ast(\alpha)\colon\rmHl_\ast\bigl({H}, \rmL^1(X_{G},\bbR)\bigr)\to
		\rmHl_\ast\bigl({G},\rmL^1(X_{H},\bbR)\bigr).
	\]
	Furthermore, the dual of the map~\eqref{eq:induction map on coinvariants} is
	the map~\eqref{eq:induction map on invariants}.
\end{theorem}
	
	We need the following general (and easy) lemma
	\begin{lemma}\label{lem:coarea formula}
	Let $(X,\nu_X)$ and $(Y,\nu_Y)$ be standard
	measure spaces and $p\colon X\to Y$ a measurable map such that
	\begin{enumerate}
	\item the fiber
	$p^{-1}(y)$ is countable for $\nu_Y$-a.e.~$y\in Y$ and
	\item $p$ is locally measure-preserving,
	that is, if $p\vert_A$ is injective for a measurable $A\subset X$,
	then $\nu_X(A)=\nu_Y(p(A))$.
	\end{enumerate}
	Then for any $f\in\rmL^1(X,\nu_X)$
	the function $y\mapsto \sum_{x\in p^{-1}(y)}f(x)$ is $\nu_Y$-integrable and
	\begin{equation*}%\label{eq:discrete coarea formula}
	\int_X fd\nu_X=\int_Y\sum_{x\in p^{-1}(y)}f(x)d\nu_Y(y).
	\end{equation*}
	\end{lemma}
	\begin{proof}
	The assertion is obvious for $f=\chi_A$ being the characteristic function of a measurable
	subset $A\subset X$ for which $p\vert_A$ is injective. By the selection theorem
	every $f\in\rmL^1(X,\nu_X)$ can be approximated by linear combinations of such
	characteristic functions which proves the lemma.
	\end{proof}
\begin{proof}[Proof of Theorem~\ref{thm:induction in homology}]
    We verify that the dual of map~\eqref{eq:induction map on coinvariants} is
	the map~\eqref{eq:induction map on invariants}, that is,
	for $f\in C^{(1)}_k(H)\widehat\otimes_{\bbZ H}\rmL^1(X_{G},\bbR)$ and
	$g\in C_b^k(G,L^\infty(X_H,\bbR))$,
	\begin{equation}\label{eq: adjoint relation}
		\langle \alpha_kf, g\rangle =\langle f, \alpha^k g\rangle
	\end{equation}
	holds true. Since $\phi_k^\alpha$ is countable-to-one and locally measure-preserving,
	\eqref{eq: adjoint relation} is implied by the previous lemma as follows:
   	\begin{align*}
		\langle \alpha_kf, g\rangle &=
		\frac{\mu(X_{H})}{\mu(X_{G})}\sum_{\bar{{g}}\in{G}^{k+1}}\int_{X_{H}}
			\sum_{\substack{(\bar{{h}},y)\in{H}^{k+1}\times X_{G}\\\phi_k^\alpha(\bar{{h}},y)=(\bar{{g}},x)}}f(\bar{{h}},y)g(\bar{g},x)\mu(X_{H})^{-1}d\mu(x)\\
		&=\mu(X_{G})^{-1}\sum_{\bar{{g}}\in{G}^{k+1}}\int_{X_{H}}
					\sum_{\substack{(\bar{{h}},y)\in{H}^{k+1}\times X_{G}\\\phi_k^\alpha(\bar{{h}},y)=(\bar{{g}},x)}}f(\bar{{h}},y)g\circ\phi_k^\alpha(\bar{h},y)d\mu(x)\\
		&=\mu(X_{G})^{-1}\cdot\sum_{\bar{{h}}\in{H}^{k+1}}\int_{X_{G}}
					f(\bar{{h}},x)g\circ\phi_k^\alpha(\bar{{h}},x)d\mu(x)\\
		&=\langle f, \alpha^k g\rangle
\end{align*}

Since $\alpha^\ast$ is a chain map, we know that the dual of
$\alpha_kd-d\alpha_{k+1}$ vanishes. The Hahn-Banach theorem implies that
$\alpha_kd-d\alpha_{k+1}=0$, so $\alpha_\ast$ is also a chain map. Similary
one concludes that $\alpha_\ast$ descends to the coinvariants from the fact that
$\alpha^\ast$ restricts to the invariants.

Since $\rmHb^k(\alpha)$ is an isometric isomorphism, also
$\rmHl_k(\alpha)$ is an isometric isomorphism by~\cite{loeh-iso}*{Theorem~1.1}.
\end{proof}

\begin{remark} With more effort one can also show that $\rmHl_\ast(\alpha)$ does
	not depend on the choice of the fundamental domains, thus could be rightfully
	denoted by $\rmHl_\ast(\Omega)$ similar to the cohomological case. Since
	we do not need this, we refrain from proving this.
\end{remark}

\subsection{Invariance of the simplicial volume}\label{sub:invariance simplicial volume}

\begin{lemma}\label{lem:factorization over sobolev}
  Let $(\Omega,\mu)$ be an ME-coupling of two finitely generated groups
  ${G}$ and ${H}$.
  Let $X_{G}$ and $X_{H}$ be fundamental domains
  of the ${G}$- and ${H}$-action, respectively.
  Let $c_\Omega=\mu(X_H)/\mu(X_G)$ be the coupling index.
  Assume that $X_{G}$ is
  integrable, and let $\alpha\colon{H}\times X_{G}\to {G}$ be the
  associated integrable cocycle.
  Then the image of the composition
\begin{equation}\label{eq:cocycle map}
	H_\ast\bigl({H}, \bbZ\bigr)\to
	H_\ast^{(1)}\bigl({H}, \rmL^1(X_{G},\bbR)\bigr)\xrightarrow{H_\ast^{(1)}(\alpha)}
	H_\ast^{(1)}\bigl({G}, \rmL^1(X_{H},\bbR)\bigr),
\end{equation}
  is contained in
   \begin{equation}\label{eq:inclusion map}
		c_\Omega\cdot \im\Bigl(H_\ast^{(1,1)}\bigl({G},\rmL^1(X_{H},\bbZ)\bigr)\to
		H_\ast^{(1)}\bigl({G},\rmL^1(X_{H},\bbR)\bigr)\Bigr),
   \end{equation}
	where all maps except $H_\ast^{(1)}(\alpha)$ are the composition of
	the corresponding comparision and
	  coefficient change maps.
\end{lemma}

\begin{proof}
  For a $(k+1)$-tuple
  $\bar{{h}}=({h}_0,\dots,{h}_k)\in{H}^{k+1}$ we abbreviate:
  \begin{gather*}
    \alpha(\bar{{h}},x)=\bigl(\alpha({h}_0,x),\dots,\alpha({h}_k,x)\bigr),\\
    \bar{{h}}^{-1}=\bigl({h}_0^{-1},\dots,{h}_k^{-1}\bigr).
  \end{gather*}
We use the identifications in~Example~\ref{ex:L1-C}.
The image of
$\rmC_k({H})=\rmC_k(H)\otimes_\bbZ\bbZ$ in
\[C_k^{(1)}(H)\widehat\otimes_{\bbZ} L^1(X_G,\bbZ)\cong
\rmL^1({H}^{k+1}\times X_{G},\bbZ)\]
is certainly contained in the set of \emph{bounded}
measurable functions $f:{H}^{k+1}\times X_{G}\to\bbZ$ for which
there is a finite subset $F\subset{H}^{k+1}$ such that $f$ is
supported on $F\times X_{G}$.
Let $f:{H}^{k+1}\times X_{G}\to\bbZ$ be such.
It is immediate from~\eqref{eq:induction map on coinvariants} that
$c_\Omega^{-1}\cdot \alpha_k f$ is $\bbZ$-valued. So it remains to show that
 \[
 	\int_{{G}^{k+1}\times X_{H}}\abs{\alpha_kf(\bar{g},y)}
\diam(\bar{g})<\infty.
 \]

 Using Lemma~\ref{lem:coarea formula} this is implied by
 \begin{align*}
	\int_{{G}^{k+1}\times X_{H}}\abs{\alpha_kf(\bar{g},y)}\diam(\bar{g})
	&\le \int_{{G}^{k+1}\times X_{H}}\int_{(\phi_k^\alpha)^{-1}(\bar{g},y)}\abs{f(\bar{h},x)}\diam(\bar{g})\\
	&= \int_{{G}^{k+1}\times X_{H}}\int_{(\phi_k^\alpha)^{-1}(\bar{g},y)}\abs{f(\bar{h},x)}\diam\bigl(\bar\alpha(\bar{h}^{-1},x)^{-1}\bigr)\\
	&\overset{\ref{lem:coarea formula}}{=} \int_{{H}^{k+1}\times X_{G}}\abs{f(\bar{h},x)}\diam\bigl(\bar\alpha(\bar{h}^{-1},x)^{-1}\bigr)\\
	&\le \operatorname{ess-sup}(f)\cdot\int_{{H}^{k+1}\times X_{G}}\diam\bigl(\bar\alpha(\bar{h}^{-1},x)^{-1}\bigr)<\infty.
 \end{align*}
 The last step follows from the integrability.
\end{proof}

Let $N$ be an aspherical topological space, and let $H=\pi_1(N)$. By asphericity and
the fundamental lemma of homological algebra there is up to
equivariant chain homotopy a unique $H$-equivariant chain homomorphism
\[
	c_H\colon C_\ast(\widetilde N)\to C_\ast(H)
\]
from the singular chain complex of the universal cover $\widetilde N$
to the standard resolution of $H$. By a theorem
of Gromov~\cite{ivanov}*{Theorem~4.1} the map $c_H$ induces an isometric isomorphism
in bounded cohomology with $\bbR$-coefficients. By the translation principle in~\cite{loeh-iso}*{Theorem~1.1}
$c_H$ induces an isometric isomorphism in $\ell^1$-homology, and thus
(compare~\cite{loeh-iso}*{Proposition~2.4})
the induced
map in homology is an isometric isomorphism:

\begin{lemma}\label{lem:fundamental class in group homology}
Let $N$ be aspherical and $H=\pi_1(N)$. The canonical map
	\[
		H_\ast(c_H)\colon H_\ast(N,\bbR)\to H_\ast(H,\bbR)
	\]
	is an isometric isomorphism with respect to the semi-norms induced by
	the \mbox{$\ell^1$-norms}.
\end{lemma}

\begin{theorem}\label{thm:main result about induction in cohomology}
	Let $M$ and $N$ be closed, oriented, negatively curved manifolds of
	dimension $n$.
	Let $(\Omega,\mu)$ be
	an ergodic, integrable ME-coupling $(\Omega,\mu)$ of the fundamental groups
	${G}=\pi_1(M)$ and ${H}=\pi_1(N)$ with coupling
	index $c=\frac{\mu({H}\bs\Omega)}{\mu({G}\bs\Omega)}$.
	
	Let $x_G\in H^n(G,\bbR)$ be the element that maps to the
	cohomological fundamental class
	of $M$ under the isomorphism $H^n(c_G): H^n(G)\to H^n(M)$. Define
	$x_H\in H^n(H,\bbR)$ analogously.
	
	Suppose that $x_G^b\in H_b^n({G},\bbR)$ is an element that maps
    to $x_G$ under the comparision (forgetful) map $H_b^n({G},\bbR)\to H_n({G},\bbR)$.
	Consider the
	composition
	\begin{multline}\label{eq:composition induction cohomological}
		H_b^n({G}, \bbR)\to H_b^n(G,\rmL^\infty({G}\bs\Omega,\bbR))
	    \xrightarrow{H_b^n(\Omega)}H_b^n({H},\rmL^\infty({G}\bs\Omega,\bbR))\\
     	\xrightarrow{I^n} H_b^n({H},\bbR)\to H^n({H},\bbR)
	\end{multline}
	where the first map is induced by the inclusion of constant functions,
	$I^n$ is the map induced by integration in the
	coefficients and the last map is the comparision map. Then $x_G^b$ is
	mapped to $\pm c\cdot x_H$
	under~\eqref{eq:composition induction cohomological}.
\end{theorem}

\begin{proof}[Proofs of Theorems~\ref{thm:invariance of simplicial volume} and~\ref{thm:main result about induction in cohomology}]
    Let
	$H_n(i_H)\colon H_n({H},\bbR)\to H_n^{(1)}({H}, \bbR)$ denote the comparision
	map; it is isometric
	with respect to the semi-norms induced by the $\ell^1$-norm on the chain complexes:
	This follows from the fact that
	$C_\ast({H})\otimes_{\bbZ{H}}\bbR\to C_\ast^{(1)}({H})\widehat\otimes_{\bbZ H}\bbR$ is
	isometric and has dense image (compare~\cite{loeh-iso}*{Proposition~2.4}).
    We denote -- by a slight abuse of notation -- the comparision (forgetful)
    map for the group $H$ in
    bounded cohomology by $H^n(i_H)\colon H^n_b(H,\bbR)\to H^n(H,\bbR)$.
	We define $H_n(i_G)$ and $H^n(i_G)$ for the group $G$ similarly.

	Let $X_H\subset\Omega$ and $X_G\subset\Omega$ be integrable fundamental domains of
	the $H$-action and $G$-action on $\Omega$, respectively. Let
	$\alpha\colon H\times X_G\to G$ be the cocycle associated to $X_G$. For the
	following we endow $X_H$ and $X_G$ with the normalized measures
	$\mu(X_H)^{-1}\mu\vert_{X_H}$ and $\mu(X_G)^{-1}\mu\vert_{X_G}$, respectively.

	With normalization, the chain map
	$j_H\colon C_\ast^{(1)}({H})\widehat\otimes_{\bbZ H}\bbR\to
	C_\ast^{(1)}({H})\widehat\otimes_{\bbZ H}\rmL^1(X_{G},\bbR)$
	given by the inclusion of constant functions is isometric.
	Integration in $L^1(X_G,\bbR)$ provides a norm-decreasing left inverse.
	Hence the induced map in $\ell^1$-homology
	\[H_n^{(1)}(j_{H})\colon H_n^{(1)}\bigl({H},\bbR\bigr)\to
	 H_n^{(1)}\bigl({H},\rmL^1(X_{G},\bbR)\bigr)\]
	is isometric. Again by a slight abuse of notation, we denote
	the map in bounded cohomology induced by inclusion of constants maps by
	\[
		H^n_b(j_H)\colon H^n_b\bigl(H,\bbR\bigr)\to
		H^n_b\bigl(H,\rmL^\infty(X_{G},\bbR)\bigr).
	\]
	We define the map $j_{G}$ for the group ${G}$ similarly.
	
	We start with the proof of Theorem~\ref{thm:invariance of simplicial volume}.
	Let $m=\dim(M)$ and $n=\dim(N)$. Assume that $\norm{N}>0$.
	Let $[N]\in H_n(N,\bbR)$ be the homological fundamental class of $N$.
	Since each map in the
	composition
	\begin{multline*}
		H_n({H},\bbR)\xrightarrow{H_n(i_H)}
		H_n^{(1)}({H},\bbR)\xrightarrow{H_n^{(1)}(j_{H})}
	    H_n^{(1)}({H},\rmL^1(X_{G},\bbR))\to\\
		\xrightarrow{H_n^{(1)}(\alpha)}
		H_n^{(1)}({G},\rmL^1(X_{H},\bbR))
	\end{multline*}
	is isometric with respect to the semi-norms induced by the respective
	$\ell^1$-norms (see Theorem~\ref{thm:induction in homology})
	and $H_n({H},\bbR)$ is
	generated by the element $H_n(c_{H})([N])$ with positive semi-norm
	(Lemma~\ref{lem:fundamental class in group homology}), we obtain
	that $H_n^{(1)}(\alpha)\circ H_n^{(1)}(j_{H})\circ  H_n(i_H)$ is injective.
	Lemma~\ref{lem:factorization over sobolev} and the fact
	that $[N]\in\im(H_n(N,\bbZ)\to H_n(N,\bbR))$ yield that
	\begin{multline*}
		0\ne H_n^{(1)}(\alpha)\circ H_n^{(1)}(j_{H})\circ  H_n(i_H)\circ  H_n(c_H)([N])
		\in\\ c_\Omega\cdot\im\Bigl(H_n^{(1,1)}\bigl({G},\rmL^1(X_{H},\bbZ)\bigr)\to
		H_n^{(1)}\bigl({G},\rmL^1(X_{H},\bbR)\bigr)\Bigr).
	\end{multline*}
	If ${G}$ is word-hyperbolic, then Theorem~\ref{hyp-sob} implies
	that
	\begin{multline}\label{eq:image essentially in ordinary homology}
		H_n^{(1)}(\alpha)\circ H_n^{(1)}(j_{H})\circ  H_n(i_H)\circ  H_n(c_H)([N])
		\in\\ c_\Omega\cdot\im\Bigl( H_n\bigl({G},\rmL^1(X_{H},\bbZ)\bigr)\to
			H_n^{(1)}\bigl({G},\rmL^1(X_{H},\bbR)\bigr)\Bigr).
	\end{multline}
    In particular, $H_n\bigl({G},\rmL^1(X_{H},\bbZ)\bigr)\ne 0$, which implies
	that $n\le m=\dim(M)$.
	
	Next assume that ${H}$ and ${G}$ are both word-hyperbolic and
	that $M$ and $N$ have positive simplicial volume. From the argument above and
	by symmetry we conclude that $m=n$.
	
	 The group ${G}$ is an orientable Poincare duality group; the
	  Poincare duality isomorphism is functorial with respect to
	  coefficient homomorphisms. Further, for any coefficient module $E$
	  there is a functorial isomorphism $\rmH^0({G},E)\cong
	  E^{G}$. Thus we obtain a commutative diagram:
	  \[
	 	\xymatrix{
			\rmH_n\bigl({G},\bbZ\bigr)\ar[r]\ar[d]^\cong&
			\rmH_n\bigl({G},\rmL^1(X_{H},\bbZ)\bigr)\ar[d]^\cong\\
	    	\rmH^0\bigl({G},\bbZ\bigr)\ar[r]\ar[d]^\cong&
			\rmH^0\bigl({G},\rmL^1(X_{H},\bbZ)\bigr)\ar[d]^\cong\\
	    	\bbZ\ar[r]^\cong &
			\rmL^1(X_{H},\bbZ)^{G}
		}
	  \]
	  The bottom map is an isomorphism because of ergodicity.
	  In combination with~\eqref{eq:image essentially in ordinary homology} this
	  implies that there is a non-zero integer $e\in\bbZ$ such that
	  \begin{multline}\label{eq:image of fundamental class}
	  			H_n^{(1)}(\alpha)\circ H_n^{(1)}(j_{H})\circ H_n(i_H)\circ H_n(c_H)([N])\\=e\cdot c_\Omega\cdot
				H_n^{(1)}(j_{G})\circ H_n(i_{G})\circ  H_n(c_{G})([M]).
	  \end{multline}
	  Since the maps involved here are isometric and $\abs{e}\ge 1$, this implies
	  that
	  \[
	  		\norm{N}\ge c_\Omega\cdot\norm{M}.
	  \]
	  By interchanging the roles of ${H}$ and ${G}$, we obtain
	  similarly $\norm{M}\ge c^{-1}\cdot\norm{N}$ and thus
  	\begin{equation}\label{eq:e is one}
		e=\pm 1~\text{ and }~\norm{M}=c_\Omega\cdot\norm{N},
	\end{equation}
	which concludes the proof
	of Theorem~\ref{thm:invariance of simplicial volume}.
	
	Next we prove Theorem~\ref{thm:main result about induction in cohomology}.
 	The assumptions imply that $M$ and $N$ have positive simplicial
	volume~\cite{gromov}*{0.3~Thurston's theorem}. Hence we know from the
	argument above that~\eqref{eq:image of fundamental class} holds with
	$e=\pm 1$. The assertion follows from the fact that $H^n(H,\bbR)\cong\bbR$
	and $H^n(i_H)\circ I^n\circ H_b^n(\Omega)\circ H_b^n(j_G)(x_G^b)$
	evaluated against the image $H_n([c_H])([N])$
	of the fundamental class of $[N]$ is $\pm c_\Omega$:
	\begin{align*}%\label{eq:induction assertion to prove}
		\bigl\langle H^n(i_H)\circ I^n\circ H_b^n(\Omega)&\circ H_b^n(j_G)(x_G^b), H_n(c_H)([N])\bigl\rangle\\ &=
			\bigl\langle H^n(i_H)\circ I^n\circ H_b^n(\alpha)\circ H_b^n(j_G)(x_G^b), H_n(c_H)([N])\bigl\rangle \\
		&= 	\bigl\langle I^n\circ H_b^n(\alpha)\circ H_b^n(j_G)(x_G^b), H_n(i_H)\circ H_n(c_H)([N])\bigl\rangle \\
		&=\bigl\langle H_b^n(\alpha)\circ H_b^n(j_G)(x_G^b), H_n^{(1)}(j_H)\circ H_n(i_H)\circ H_n(c_H)([N])\bigl\rangle \\
		&=  \bigl\langle  H_b^n(j_G)(x_G^b),  H_n^{(1)}(\alpha)\circ H_n^{(1)}(j_H)\circ H_n(i_H)\circ H_n(c_H)([N])\bigl\rangle\\
\text{use~\eqref{eq:image of fundamental class}}\qquad		&=\pm c_\Omega\cdot\bigl\langle  H_b^n(j_G)(x_G^b),  H_n^{(1)}(j_G)\circ H_n(i_G)\circ H_n(c_G)([M])\bigl\rangle \\
		&=\pm c_\Omega\cdot \bigl\langle x_G^b,  H_n(i_G)\circ H_n(c_G)([M])\bigl\rangle \\
		&=\pm c_\Omega\cdot \bigl\langle x_G, H_n(c_G)([M])\\
		&= \pm c_\Omega.\qedhere
	\end{align*}
\end{proof}

\end{document}